\documentclass[12pt]{amsart}

\usepackage{amsmath,amsthm,amsfonts,amssymb,bm }

\newcommand{\be}{\begin{eqnarray}}
\newcommand{\ee}{\end{eqnarray}}
\newcommand{\ben}{\begin{eqnarray*}}
\newcommand{\enn}{\end{eqnarray*}}

\newcommand{\mi}{{m^{(1)}}}
\newcommand{\mii}{{m^{(2)}}}
\newcommand{\rhi}{{\rho^{(1)}}}
\newcommand{\rhii}{{\rho^{(2)}}}
\newcommand{\n}{{\mathbb{R}^n}}
\newcommand{\nn}{{\mathbb{R}^{2n}}}
\newcommand{\supp}{\text{supp}}

\newtheorem{theorem}{\textbf Theorem}[section]
\newtheorem{Lemma}{\textbf Lemma}[section]

\newtheorem{corollary}{\textbf Corollary}[section]

\newtheorem{definition}{\textbf Definition}[section]

\allowdisplaybreaks

\def\gfz{\genfrac{}{}{0pt}{}}

\newcounter{remark}
\newenvironment{remark}%
{\par \stepcounter{remark} {\it Remark
\arabic{section}.\arabic{remark}.}~}

\allowdisplaybreaks

\begin{document}

\title{ $L^p$ Boundedness of rough  Bi-parameter Fourier Integral Operators}
\thanks {The first author's research was partly supported
by a grant of NNSF of China  (No. 11371056) and the second and third authors' research was partly supported  by  a US NSF grant (DMS-1301595) and the third author was partly supported by a Simons Fellowship from the Simons Foundation. Corresponding Author: Guozhen Lu at gzlu@wayne.edu. }

\author{Qing Hong}
\address{School  of Mathematical Sciences\\
 Beijing Normal  University\\
  Beijing 100875, China \\
Email: hongqingjx@126.com}
\author{Guozhen Lu }
\address{Department of Mathematics\\
Wayne State University\\
Detroit, MI 48202, USA\\
Email: gzlu@wayne.edu}
\author{Lu Zhang}
\address{Department of Mathematics\\
Wayne State University\\
Detroit, MI 48202, USA\\
Email: eu4347@wayne.edu}
\date{}

\maketitle

\begin{abstract}

In this paper, we  will investigate the boundedness of the bi-parameter Fourier integral operators (or FIOs for short) of the following form:
$$T(f)(x)=\frac{1}{(2\pi)^{2n}}\int_{\mathbb{R}^{2n}}e^{i\varphi(x,\xi,\eta)}\cdot a(x,\xi,\eta)\cdot\widehat{f}(\xi,\eta)d\xi d\eta,$$
 where for $x=(x_1,x_2)\in \mathbb{R}^{n}\times \mathbb{R}^{n}$ and $\xi,\eta \in \mathbb{R}^{n}\setminus\{0\}$,
the amplitude $a(x,\xi,\eta)\in L^\infty BS^m_\rho$ and the phase function is of the form
$
\varphi(x,\xi,\eta)=\varphi_1(x_1,\xi)+\varphi_2(x_2,\eta)$ with $\quad \varphi_1,\varphi_2 \in L^\infty \Phi^2 (\mathbb{R}^{n}\times\mathbb{R}^{n}\setminus\{0\})$ and
   $\varphi(x, \xi, \eta)$ satisfies a certain rough non-degeneracy condition $\eqref{RND}$.

The study of these operators are motivated by the $L^p$ estimates for one-parameter FIOs and bi-parameter Fourier multipliers  and pseudo-differential operators. We will first define the bi-parameter FIOs  and then study the $L^p$ boundedness of such operators  when their phase functions have compact support in frequency variables with certain necessary non-degeneracy conditions. We will then establish the $L^p$ boundedness of the more general FIOs with amplitude $a(x,\xi,\eta)\in L^\infty BS^m_\rho$ and non-smooth phase function $\varphi(x,\xi,\eta)$ on $x$ satisfying a rough non-degeneracy condition.

\medskip

\noindent
{\bf Keywords:} Bi-parameter Fourier integral operators, Seeger-Sogge-Stein decomposition, $L^p$ boundedness, non-smooth amplitude and phase functions, non-degeneracy condition.

\noindent
{\bf Mathematics Subject Classification:} 35S30; 42B20.
\medskip

{\bf
\noindent
\date{\today}}

\end{abstract}
\section{\bf Introduction}

L. H\"{o}rmander \cite{H} defined the Fourier integral operator (FIO)  $T$ in the following form
$$Tf(x)=\int a(x,\xi)\widehat{f}(\xi)e^{i\varphi(x,\xi)}d\xi, $$
for $f$ in the class of Schwartz functions $\mathcal{S}(\n)$,
where $x \in \n $ is the spatial variable, $\xi\in \n$ is the frequency variable, $a$ is the amplitude function and $\varphi$ is the phase function. In the study of FIOs, we often assume  $a\in S^m_{\rho,\delta}$, that is, a collection of smooth functions that satisfy
\begin{equation}\label{amplitude-oneparameter}
\big|\partial_x^\alpha\partial_\xi^\beta a(x,\xi)\big|
\leq C_{\alpha,\beta}(1+|\xi|)^{m+\delta|\alpha|-\rho|\beta|},\quad a\in C^\infty(\n \times \n)
\end{equation}
for $m\in\mathbb{R}$, $\rho,\delta\in[0,1]$ and all multi-indices $\alpha$ and $\beta$. The phase function $\varphi\in C^\infty$ is homogeneous of degree 1 in $\xi$ and satisfies the non-degeneracy condition, that is the modulus of the determinant of the mixed Hessian of the phase does not vanish. \\

The local $L^2$ boundedness of FIOs with non-degenerate phase functions was investigated by G. Eskin \cite{E} for $a\in S^0_{1,0}$ and by L. H\"{o}rmander \cite{H}  for $a\in S^0_{\rho,1-\rho},\ \rho\in[1/2,1]$.  A. Seeger, C. Sogge and E. Stein \cite{SSS} further established  the local $L^p$ $ (1< p<\infty)$ boundedness of smooth FIOs with non-degenerate and homogeneous $\varphi$ for $a\in S^m_{\rho,1-\rho}$ compactly supported in $x$, provided that $ \ \rho\in[1/2,1],\ m\leq (\rho-n)\big|\frac1p-\frac12\big|$. For more extensive study of local boundedness of FIOs, we refer to the book of C. Sogge \cite{S} and references therein. \\

For the global $L^2$ boundedness of FIOs when $\varphi\in C^\infty(\n \times \n\setminus \{0\})$  is homogeneous and $a\in S^0_{0,0}$, see e.g.   D. Fujiwara \cite{F}.  Applications to smoothing estimates for evolution partial differential equations require  non-smooth  phases, in addition to minimizing the decay assumptions on the regularity of symbols, see  the works of M. Ruzhansky and M. Sugimoto \cite{RS2} and  more general weighted Sobolev $L^2$ estimates   given by the same authors  in  \cite{RS3}\\

The global $L^p$ boundedness (when $a$ is in the so called SG classes) was established by E. Cordero, F. Nicola and L. Rodino in \cite{CNR}. Moreover, for the general amplitudes $a$ from the classes $ L^p S^m_{\rho,\delta}$ where $\rho,\delta\in[0,1]$, which depends on the growth/decay order of the amplitude in $x$ and $y$ variables,  S. Coriasco and M. Ruzhansky \cite{CR}, and Rodr\'{i}guez-L\'{o}pez and W. Staubach \cite{RS} proved  the $L^p$ estimate of the rough FIOs  with non-smooth amplitude on $x$ and smooth phases. The global $L^p$ boundedness of the rough FIOs with non-smooth phases $\varphi\in L^\infty\Phi^2$ was carried out by D. Dos Santos Ferreira and W. Staubach \cite{RS}. We refer the reader to Section 2 for definitions of the classes $ L^p S^m_{\rho,\delta}$ for the amplitudes $a$  and
$ L^\infty\Phi^2$ for the phase functions $\varphi$.
\medskip

In the work of Seeger, Sogge and Stein \cite{SSS}, the following boundedness  of Fourier integral operators (FIOs) was established.

\begin{theorem}\label{th-SSS}
Suppose that $1<p<\infty$ and $m\leq-(n-\rho)\big|\frac1p-\frac12\big|,\rho\in[1/2,1]$. Assume also that the amplitude function $a(x,\xi)\in S^m_\rho$ and the phase function $\varphi(x,\xi)\in \Phi^2$ satisfy the strong non-degeneracy condition (or SND for short) \eqref{SND}. Then we have that the FIO
\begin{equation}
\label{smooth-fio}
Tf(x)=\int a(x,\xi)\widehat{f}(\xi)e^{i\varphi(x,\xi)}d\xi
\end{equation}
is a bounded operator from $L^p_{comp}$ to $L^p_{loc}$.
\end{theorem}

In \cite{RS},   global boundedness of FIOs was established by S. Rodr\'{i}guez-L\'{o}pez and W. Staubach when the phase function $\varphi(x,\xi)$ and the symbol $a(x,\xi)$ are not necessarily smooth with respect to $x$ (see Kenig and Staubach \cite{KS} for such type of global $L^p$ estimates for pseudo-differential operators with non-smooth amplitude).  

\begin{theorem}{$($\cite{RS}$)$}
\label{th-FS}
Suppose that the amplitude function $a(x,\xi)\in L^\infty S^m_\rho$ and the phase function $\varphi(x,\xi)\in L^\infty \Phi^2$ satisfy the rough non-degeneracy condition (or RND for short) \eqref{RND}. Then the FIO
\begin{equation}
\label{smooth-fio}
Tf(x)=\int a(x,\xi)\widehat{f}(\xi)e^{i\varphi(x,\xi)}d\xi
\end{equation}
is a bounded operator from $L^p$ to $L^p$ $(1\leq p\leq \infty)$ provided that
\begin{itemize}
  \item[(i)] $m<\frac{n(\rho-1)} {p}-\frac{(n-1)}{2p}$ when $1\leq p\leq 2$,
  \item[(ii)] $m<\frac{n(\rho-1)}{2}-\frac{n-1}{2}\big(1-{1\over p}\big)$ when $2\leq p \leq \infty$.
\end{itemize}
\end{theorem}

\vskip0.3cm

Motivated by these  works on $L^p$ estimates for one-parameter FIOs and the $L^p$ estimates for multi-parameter singular integral operators (see e.g., R. Fefferman and E. M. Stein \cite{FS} and Journ\'e \cite{Jo}),  and more recent works of the $L^p$ estimates for multi-parameter Coifman-Meyer Fourier multipliers of Muscalu, Pipher, Tao and Thiele \cite{MPTT1, MPTT2} (see also \cite{CL}),  and the $L^p$ estimates for multi-parameter pseudo-differential operators (see \cite{MS}, \cite{DL}, \cite{HL}),
our main goal in this paper is to study the $L^p$ estimates for  bi-parameter FIOs with the non-smooth phases and amplitudes with respect to $x$. That is, we will study the operator
\begin{equation}
\label{opdef}
T(f)(x)=\int_{\mathbb{R}^{2n}}e^{i\varphi(x,\xi,\eta)}\cdot a(x,\xi,\eta)\cdot\widehat{f}(\xi,\eta)d\xi d\eta,
\end{equation}
where $x=(x_1,x_2)\in \mathbb{R}^{n}\times \mathbb{R}^{n}$, $\xi,\eta \in \mathbb{R}^{n}\setminus\{0\}$ and
\begin{equation}
\label{phasedef}
\varphi(x,\xi,\eta)=\varphi_1(x_1,\xi)+\varphi_2(x_2,\eta), \quad \varphi_1,\varphi_2 \in L^\infty \Phi^2 (\mathbb{R}^{n}\times\mathbb{R}^{n}\setminus\{0\})
\end{equation}
 with $a(x,\xi,\eta)\in L^\infty BS^m_\rho$ as defined in Definition 2.3 in Section 2 and $\varphi_1(x_1, \xi)$ and $\varphi(x_2, \eta)$ satisfy the rough non-degeneracy condition $\eqref{RND}$.
 (See Section \ref{def-th} for definitions and the notations used here.)

\vskip0.5cm

We now make some remarks on the assumptions on the phase functions $a(x, \xi, \eta)$ and amplitudes $\varphi(x, \xi, \eta)$ in the bi-parameter setting of FIOs and explain that these amplitudes and phase functions are not necessarily covered in the classical case of one-parameter FIOs.

\begin{remark}
By Definition 2.3, it is easy to see that $a(x, \xi,\eta)\in  L^\infty BS^m_\rho$ satisfies weaker condition than the assumption on the amplitude in (\ref{amplitude-oneparameter}) in the one-parameter setting. Therefore, the bi-parameter  FIOs we are considering in this paper indeed covers a wider class of amplitudes than those in the one-parameter case.

 \end{remark}
\medskip

\begin{remark}
  The assumption for the phase functions $\varphi(x,\xi,\eta)$ in $\eqref{phasedef}$  are given in a way where variables are separated in different parameters. We will see such an assumption is necessary. Recalling that  in the study of the single parameter Fourier integral operators, the phase functions $\varphi(x,\xi)$ are required to be positively homogeneous of degree $1$ in $\xi$ so that  Euler's theorem can be used. Also, the phase functions need to satisfy some non-degeneracy conditions, as in Definition $\ref{Strong-Nondegeneracy}$ or $\ref{Nonsmooth-Nondegeneracy}$. In our bi-parameter setting, similar conditions are needed. On one hand, we need to make the phase functions positively homogeneous of degree $1$ in both $\xi$ and $\eta$ in order to use  Euler's theorem. On the other had, we also need to use the non-degeneracy conditions in separate variables. Therefore, it is necessary to make the phase functions defined in $\eqref{phasedef}$ as the sum of two functions in different variables. Moreover, there are phase functions satisfying our conditions but not those  used in the single parameter setting. For example, the phase functions in the single parameter version of Theorem $\ref{th-e}$ on $\mathbb{R}^{2n}$ should satisfy
  \begin{eqnarray}\label{singlecondition}
  \sup\limits_{(\xi,\eta) \in \mathbb{R}^{2n}\setminus\{(0,0)\}}(|\xi|+|\eta|)^{-1+|\alpha_1|+|\alpha_2|}\big\|\partial_{\xi}^{\alpha_1}\partial_\eta^{\alpha_2}\varphi(x,\xi,\eta)\big\|_{L^\infty(\mathbb{R}^{2n})}< \infty.
  \end{eqnarray}
  for all multi-indices $\alpha_1,\alpha_2$ with $|\alpha_1|+|\alpha_2|\geq 2$.

  However, in bi-parameter setting,   $\eqref{phasedef}$ implies
  \begin{equation} \label{bicondition}
   \sup\limits_{\substack{\xi \in \mathbb{R}^n\setminus\{0\} }} \sup\limits_{\substack{\eta \in \mathbb{R}^n\setminus\{0\} }}|\xi|^{-1+|\alpha_1|} \big\|\partial_{\xi}^{\alpha_1}\varphi(x,\xi,\eta)\big\|_{L^\infty(\mathbb{R}^{2n})}<\infty.
  \end{equation}
  for all  multi-indices $\alpha_1$ satisfying $|\alpha_1|\geq 2$. Note that the condition $\eqref{bicondition}$ is actually weaker than $\eqref{singlecondition}$ for $|\alpha_1|\geq 2, \, |\alpha_2|=0$.

Therefore,  there are  phase functions $\varphi(x, \xi, \eta)$ in the bi-parameter FIOs that are not included in those considered 
in the one-parameter FIOs.

\end{remark}

\vskip0.5cm
The main results of this paper are as follows:

\begin{theorem}\label{th-a}
Let $0<r\leq\infty$ and $ 1\leq p,q\leq\infty$ satisfy $\frac1p+\frac1q=\frac1r$. If the amplitude function $a(x,\xi,\eta)\in L^p BS^m_\rho$ for $m\leq0$, $0\leq \rho\leq 1$ are compactly supported on $\xi$ and $\eta$ and the phase functions $\varphi_1(x_1,\xi), \varphi_2(x_2,\eta)\in \Phi^2$ satisfy the strong non-degeneracy condition (\ref{SND}). Then the biparameter FIO
$$T(f)(x)=\int_{\mathbb{R}^n\times\mathbb{R}^n}e^{i\varphi_1(x_1,\xi)}e^{i\varphi_2(x_2,\eta)} a(x,\xi,\eta)\cdot\widehat{f}(\xi,\eta)d\xi d\eta$$
is bounded from $L^q$ to $L^r$.
\end{theorem}

\begin{remark}
\begin{itemize}
  \item[(i)] The above theorem requires the smoothness of  phase $\varphi(x,\xi,\eta)$ with respect to $x$, but allows non-smooth amplitude $a(x,\xi,\eta)$.
  \item[(ii)] In particular when $p=\infty$, $T$ is bounded on $L^q \ (1\leq q\leq\infty)$. From theorems below, we will see the $L^q$ estimate still holds with non-smooth phase $\varphi$ which satisfy the rough non-degeneracy condition \eqref{RND} instead.
\end{itemize}

\end{remark}

\vskip0.2cm

\begin{theorem}\label{th-e}
Let T be a bi-parameter FIO:
$$T(f)(x)=\int_{\mathbb{R}^n\times\mathbb{R}^n}e^{i\varphi_1(x_1,\xi)} e^{i\varphi_2(x_2,\eta)}\cdot a(x,\xi,\eta)\cdot\widehat{f}(\xi,\eta)d\xi d\eta$$
with amplitude function $a(x,\xi,\eta)\in L^\infty BS^m_\rho$ and phase functions $\varphi_1(x_1,\xi), \varphi_2(x_2,\eta)\in L^\infty \Phi^2$ satisfy the rough non-degeneracy condition \eqref{RND}. Then $T$ is bounded on $L^p$ $(1\leq p\leq \infty)$ provided that
\begin{itemize}
\item[$(a)$]  $m^{(i)}<\frac{n(\rho^{(i)}-1)} {p}-\frac{(n-1)}{2p}$ when $1\leq p\leq2$,
\item[$(b)$] $m^{(i)}<\frac{n(\rho^{(i)}-1)}{2}-\frac{n-1}{2}\big(1-{1\over p}\big)$ when $2\leq p\leq\infty$.\\
\end{itemize}
\end{theorem}

To prove the above Theorem \ref{th-e}, we will first give  the Seeger-Sogge-Stein decomposition
of the bi-parameter FIOs (see Section 3):
$$T(f)(x)=T_{00}(f)(x)+\sum_{\mu\nu} T^{\mu,\nu}_{jk}(f)(x).$$

We will prove the $L^p$ boundedness of $T_{00}(f)$, and then
we will prove the $L^1\rightarrow L^1,\ L^2\rightarrow L^2,\ L^\infty\rightarrow L^\infty$ boundedness properties of $T$ respectively as follows,   the interpolation argument gives  the desired  $L^p$ estimate.

\begin{theorem}\label{th-local}
Let the amplitude function $a(x,\xi,\eta)\in L^\infty BS^m_\rho$ with $m\in \mathbb{R}$ and $\rho \in [0,1]$, and the phase functions $\varphi_1(x_1,\xi), \varphi_2(x_2,\eta)\in L^\infty \Phi^2$ satisfy $\eqref{RND}$. Then for all $\Phi^1_0(\xi), \Phi^2_0(\eta)\in C_0^\infty(\mathbb{R}^n)$ supported around the origin, the bi-parameter Fourier integral operator
$$T_{00}(f)(x)=\int_{\mathbb{R}^n\times\mathbb{R}^n}e^{i\varphi_1(x_1,\xi)} e^{i\varphi_2(x_2,\eta)}\cdot \Phi^1_0(\xi) \Phi^2_0(\eta) a(x,\xi,\eta)\cdot\widehat{f}(\xi,\eta)d\xi d\eta$$
is bounded on $L^p$ for $p\in [1,\infty]$.
\end{theorem}

\begin{theorem}\label{th-b}
The bi-parameter FIO operator T  defined as in Theorem $\ref{th-e}$
is bounded on $L^1$, provided $m^{(i)}<-\frac{n-1}2-n(1-\rho^{(i)}),i=1,2$.
\end{theorem}

\begin{theorem}
\label{th-c}
The bi-parameter FIO operator T  defined as in Theorem $\ref{th-e}$
is bounded on $L^2$, provided $m^{(i)}<\frac n2(\rho^{(i)}-1)-{n-1\over 4},i=1,2$.
\end{theorem}

\begin{theorem}\label{th-d}
The bi-parameter FIO operator T  defined as in Theorem $\ref{th-e}$
is bounded on $L^\infty$, provided $m^{(i)}<-\frac{n-1}2-\frac n2(1-\rho^{(i)}),i=1,2$.
\end{theorem}
\vskip0.3cm
The paper is organized as follows:

In Section \ref{def-th}, we give some  definitions and preliminaries  that will be used in the sequel.

In Section \ref{sec-decomp}, we will recall the Littlewood-Paley decomposition and the Seeger-Sogge-Stein decomposition with some useful facts. The bi-parameter FIOs will then be decomposed in each parameter by using such decompositions.

In Section \ref{sec-thpq}, we give the proof of Theorem $\ref{th-a}$, where  the amplitude function $a(x,\xi,\eta)\in L^p BS^m_\rho$ and the phase functions $\varphi_1,\varphi_2\in \Phi^2$.

In Section \ref{th-lp}, we include the proof of the $L^p$ estimate of the bi-parameter FIOs with non-smooth phases $\varphi_1,\varphi_2\in L^\infty \Phi^2$ and amplitudes $a(x,\xi,\eta)\in L^\infty BS^m_\rho $, namely Theorem 1.4. To do this, we will divide its proof into several steps by first establishing the $L^p$ estimates of $T_{00}$ (Theorem 1.5) and then the $L^1$, $L^2$ and $L^{\infty}$ estimates for the FIO $T$.
(Theorems 1.6, 1.7 and 1.8).

\setcounter{equation}{0}
\section{\bf Some    definitions and preliminaries}
\label{def-th}

We begin with the following notations and definitions that will be needed in this paper.

\begin{definition}\label{norm}
  \begin{itemize}
 \item[$(i)$] For $\xi \in \n$, we define $\langle \xi \rangle:=(1+|\xi|^2)^{1\over 2}.$
 \item[$(ii)$] For $\xi \in \n$, we denote the annulus $\{\xi:1\leq |\xi|\leq 2\} $ in $\n$ by $A_\xi$.
\end{itemize}
\end{definition}

\begin{definition}\label{smooth-biparameter-amplitude}
A smooth function $a(x,\xi,\eta)$, where $ x=(x_1,x_2)\in \mathbb{R}^n\times \mathbb{R}^n,\xi,\eta\in\mathbb{R}^n$, is said to be in the class  $BS^m_{\rho,\delta}$ of bi-parameter amplitudes (we also say that $a$ is  a symbol of order $m$) for real pairs $m=(\mi,\mii), \rho=(\rhi,\rhii),\delta=(\delta^{(1)},\delta^{(2)})\in \mathbb{R}^2$ where $0\leq \rhi,\rhii,\delta^{(1)},\delta^{(2)}\leq 1$, if for all multi-indices $\alpha=(\alpha_1,\alpha_2),\beta=(\beta_1,\beta_2)$, there holds
$$\sup\limits_{x, \xi,\eta\in\mathbb{R}^n}(1+|\xi|)^{-m^{(1)}-\delta^{(1)}|\alpha_1|+\rho^{(1)}|\beta_1|}
(1+|\eta|)^{-m^{(2)}-\delta^{(2)}|\alpha_2|+\rho^{(2)}|\beta_2|}\big|\partial_{x_1}^{\alpha_1}
\partial_{x_2}^{\alpha_2}\partial_\xi^{\beta_1}\partial_\eta^{\beta_2}a(x,\xi,\eta)\big|<\infty.$$
\end{definition}

\begin{definition}\label{nonsmooth-biparameter-amplitude}
Let $1\leq p\leq\infty$, $m=(\mi,\mii)\in \mathbb{R}^2, \rho=(\rhi,\rhii) $ be real pairs  where $0\leq \rhi,\rhii\leq 1$. A function  $a(x,\xi,\eta), x=(x_1,x_2)\in \nn$, $\xi,\eta\in \n$ is said to belong to $L^pBS^m_\rho$, if $a(x,\xi,\eta)$ is measurable in $x$,  $a(x,\xi,\eta)\in C^\infty(\mathbb{R}_\xi^n\times\mathbb{R}_\eta^n)$ for a.e. $x\in \nn$, and for all multi-indices $\alpha=(\alpha_1,\alpha_2)$ there exists $C_\alpha>0$ such that
$$\big\|\partial_\xi^{\alpha_1}\partial_\eta^{\alpha_2}a(\cdot,\xi,\eta)\big\|_{L^p(\mathbb{R}^n\times\mathbb{R}^n)}\leq C_\alpha(1+|\xi|)^{m^{(1)}-\rho^{(1)}|\alpha_1|}
(1+|\eta|)^{m^{(2)}-\rho^{(2)}|\alpha_2|}$$
Here, for $s\in \mathbb{N}$ we define the associated semi-norm $$|a|_{p,m,s}=\sum\limits_{\gfz{\alpha=(\alpha_1,\alpha_2)}{|\alpha|\leq s}}\sup\limits_{\xi,\eta\in\mathbb{R}^n}(1+|\xi|)^{\rho^{(1)}|\alpha_1|-m^{(1)}}
(1+|\eta|)^{\rho^{(2)}|\alpha_2|-m^{(2)}}\big\|\partial_\xi^{\alpha_1}\partial_\eta^{\alpha_2}a(\cdot,\xi,\eta)\big\|_{L^p}$$
\end{definition}

\begin{definition}\label{homogeneity}[Euler's Theorem]
A real valued function $\varphi(x,\xi)\in C^\infty(\mathbb{R}^n\times\mathbb{R}^n\setminus \{0\})$ is said to be positively homogeneous of degree 1 in $\xi$, if for all $\lambda>0$, there holds$$\varphi(x,\lambda \xi)=\lambda \varphi(x,\xi). $$
Moreover, $\varphi(x,\xi)$ is positively homogeneous of degree 1 if and only if
$$ \varphi(x,\xi)=\xi\cdot\nabla_\xi \varphi(x,\xi).$$
\end{definition}

\begin{definition}\label{PhiK-phasefunction}
  A real valued function $\varphi(x,\xi)$ is said to belong to the class $\Phi^k$, if $\varphi(x,\xi)\in  C^\infty(\n \times \n\setminus \{0\})$, is positively homogeneous of degree $1$ in the frequency variable $\xi$, and for all multi-indices $\alpha$ and $\beta$ satisfying $|\alpha|+|\beta|\geq k$, there exists a positive constant $C_{\alpha,\beta}$ such that
$$\sup\limits_{(x,\xi)\in\n \times \mathbb{R}^n\setminus\{0\}}|\xi|^{-1+|\alpha|}|\partial_x^{\beta}\partial_\xi^{\alpha}\varphi(x,\xi)|\leq C_{\alpha,\beta}.$$
\end{definition}

\begin{definition}\label{Strong-Nondegeneracy}A real valued function $\varphi\in C^2(\mathbb{R}^{n}\times\mathbb{R}^{n}\setminus\{0\})$ is said to satisfy the strong non-degeneracy condition, if there exists a positive $c>0$, such that
\begin{eqnarray}\label{SND}
\bigg|\det \frac{\partial^2\varphi(x,\xi)}{\partial x_j\partial \xi_k}\bigg|\geq c
\end{eqnarray}
for all $(x,\xi)\in \mathbb{R}^{n}\times\mathbb{R}^{n}\setminus\{0\}$.
\end{definition}

\begin{definition}\label{LP-PhiK-phasefunction}
Let $1\leq p \leq \infty$. A real valued function $ \varphi(x,\xi)$ is said to belong to the class $L^p\Phi^k$, if $ \varphi$ is positively homogeneous of degree $1$, smooth on $\xi\in\n \setminus \{0\}$ ,  measurable in $x$, and for all multi-indices $\alpha$ with $|\alpha|\geq k$ there holds
$$\sup\limits_{\xi \in \mathbb{R}^n\setminus\{0\}}|\xi|^{-1+|\alpha|}\big\|\partial_\xi^{\alpha}\varphi(x,\xi)\big\|_{L^p(\mathbb{R}^{n})}\leq \infty.$$
\end{definition}

\begin{definition}\label{Nonsmooth-Nondegeneracy}
A real valued function $\varphi$ is said to satisfy  the rough non-degeneracy condition, if it is $C^1(\mathbb{R}_\xi^n)$,  bounded measurable in $x$, and there exists $C>0$, such that for any $x,y\in\n$ and $\xi\in \n\setminus\{0\}$,
\begin{eqnarray}\label{RND}
\big|(\nabla_{\xi}\varphi)(x,\xi)-(\nabla_{\xi}\varphi)(y,\xi)\big|\geq C\cdot|x-y|,
\end{eqnarray}
where $\nabla_{\xi}\varphi=(\partial_{\xi_1}\varphi,\partial_{\xi_2}\varphi,...,\partial_{\xi_n}\varphi)$.
\end{definition}
\vskip0.2cm

 Next, we will give some lemmas needed to prove our main theorems.
 We begin with the following
  bi-parameter version of the one-parameter result established in \cite{FS}.

  \begin{Lemma}
\label{le-a}
The bi-parameter FIO \,$T(f)$ defined as in $\eqref{opdef}$
with amplitude $a(x,\xi,\eta)\in L^\infty BS_{\rho,\delta}^m$ and phase function defined as in $\eqref{phasedef}$ can be written as a finite sum of operators of the form
\begin{eqnarray*}
\int_{\mathbb{R}^{2n}}e^{i(\psi_1(x_1,\xi)+\langle\nabla_{\xi}\varphi_1(x_1,\zeta_1),\xi\rangle)}\cdot e^{i(\psi_2(x_2,\eta)+\langle\nabla_{\eta}\varphi_2(x_2,\zeta_2),\eta\rangle)}\cdot a(x,\xi,\eta)\cdot\widehat{f}(\xi,\eta)d\xi d\eta
\end{eqnarray*}
where $\zeta_1,\zeta_2$ are points on the unit sphere $S^{n-1},\ \psi_1,\psi_2\in L^\infty \Phi^1$, and $a\in L^\infty BS_{\rho,\delta}^m$ is localized in the $\xi$ variable around the point $\zeta_1$, $\eta$ variable around the point $\zeta_2$.
\end{Lemma}

Next we will establish the following bi-parameter kernel estimates.

\begin{Lemma} \label{le-b}
Let $b(x,\xi,\eta)$ be a bounded function, which is $C^{2n+2}(\mathbb{R}_\xi^n\setminus \{0\}\times\mathbb{R}_\eta^n\setminus\{0\})$ and compactly supported in $\xi,\eta$. If for any $|\alpha_1|\leq n+1, |\alpha_2|\leq n+1$
$$\sup\limits_{\xi,\eta\in\mathbb{R}^n\setminus\{0\}}h_{\alpha_1}(\xi) h_{\alpha_2}(\eta)
\big\|\partial_\xi^{\alpha_1}\partial_\eta^{\alpha_2}b(\cdot,\xi,\eta)\big\|_{L^\infty}<\infty,$$
where $h_{\gamma}(z)$ is defined to be $1$ when $|\gamma|=0$ and $|z|^{-1+|\gamma|}$ otherwise.
Then, for all $0\leq\mu<1$, we have
$$\sup\limits_{\gfz{x\in\mathbb{R}^n\times\mathbb{R}^n}{u,v\in\mathbb{R}^n}}(1+|u|)^{n+\mu}\cdot(1+|v|)^{n+\mu}\bigg|\int e^{-i(u\cdot\xi+v\cdot\eta)}b(x,\xi,\eta)d\xi d\eta\bigg|<\infty$$
\end{Lemma}

\begin{proof}
The desired estimate follows easily when $|u|,|v|\leq 1$, so we consider $|u|,|v|\geq 1$, and the cases $|u|\leq 1,|v|\geq 1$ and $|v|\leq 1,|u|\geq 1$ will follow similarly.
 Assume $b(x, \xi, \eta)$ is supported in $|\xi|\leq M$ and $|\eta|\leq M$ for some $M>0$. Let $$B(x,u,v)=\int e^{-i(u\cdot\xi+v\cdot\eta)}b(x,\xi,\eta)d\xi d\eta,$$
 we have
\begin{eqnarray*}
|B(x,u,v)|&=& |u|^{-2n}|v|^{-2n}\bigg|\int e^{-i(u\cdot\xi+v\cdot\eta)}\langle u, D_\xi \rangle^n \langle  v, D_\eta \rangle^n b(x,\xi,\eta)d\xi d\eta\bigg|\\
&\leq & |u|^{-n}|v|^{-n} \bigg|\int_{|\xi|<M}\int_{|\eta|<M} \frac{1}{|\xi|^{n-1}} \frac{1}{|\eta|^{n-1}} d\xi d\eta\bigg|.
\end{eqnarray*}
 Note that the function $\beta(x,\xi,\eta)\doteq |u|^{-n}|v|^{-n} \langle u, D_\xi \rangle^n \langle  v, D_\eta \rangle^n b(x,\xi,\eta)\in C^\infty(\mathbb{R}_\xi^n\setminus \{0\}\times\mathbb{R}_\eta^n\setminus\{0\})$ satisfies

\begin{equation}\label{betacd}
\sup\limits_{\xi,\eta\in\mathbb{R}^n\setminus\{0\}}|\xi|^{n-1+|\alpha_1|}|\eta|^{n-1+|\alpha_2|}\big\|
\partial_\xi^{\alpha_1}\partial_\eta^{\alpha_2}\beta(\cdot,\cdot,\xi,\eta)\big\|_{L^\infty}<\infty,\ \ |\alpha_1|\leq 1,|\alpha_2|\leq 1.
\end{equation}

Let $\chi$ be a $C_0^\infty(\mathbb{R}^n)$ function which is one on the unit ball and zero outside the ball of radius 2, taking $0<\epsilon_1,\epsilon_2\leq1$, we have
\begin{eqnarray*}
&&|u|^{n}\cdot|v|^{n}|B(x,u,v)|=\bigg|\int e^{-i(u\cdot\xi+v\cdot\eta)}\beta(x,\xi,\eta)d\xi d\eta\bigg|\\
&&\leq\bigg |\int e^{-i(u\cdot\xi+v\cdot\eta)}\cdot\chi(\xi/\epsilon_1)\chi(\eta/\epsilon_2)\cdot\beta(x,\xi,\eta)d\xi d\eta\bigg|\\
&&\ \ \ \ \ +\bigg |\int e^{-i(u\cdot\xi+v\cdot\eta)}\cdot\chi(\xi/\epsilon_1)\big(1-\chi(\eta/\epsilon_2)\big)\cdot\beta(x,\xi,\eta)d\xi d\eta\bigg|\\
&&\ \ \ \ \ +\bigg |\int e^{-i(u\cdot\xi+v\cdot\eta)}\cdot\big(1-\chi(\xi/\epsilon_1)\big)\chi(\eta/\epsilon_2)\cdot\beta(x,\xi,\eta)d\xi d\eta\bigg|\\
&&+\bigg |\int e^{-i(u\cdot\xi+v\cdot\eta)}\big(1-\chi(\xi/\epsilon_1)\big)\big(1-\chi(\eta/\epsilon_2)\big)\beta(x,\xi,\eta)d\xi d\eta\bigg|\doteq I_1+I_2+I_3+I_4.
\end{eqnarray*}
Using $\eqref{betacd}$ we can get that:
\begin{eqnarray*}
I_1&=&\bigg |\int e^{-i(u\cdot\xi+v\cdot\eta)}\chi(\xi/\epsilon_1)\chi(\eta/\epsilon_2)\beta(x,\xi,\eta)d\xi d\eta\bigg|\leq\int_{\gfz{|\xi|\leq 2\epsilon_1}{|\eta|\leq 2\epsilon_2}}|\xi|^{1-n}|\eta|^{1-n}d\xi d\eta\leq C_0\epsilon_1\epsilon_2\\
I_2&=& |v|^{-2} \bigg |\int e^{-i(u\cdot\xi+v\cdot\eta)}\cdot\chi(\xi/\epsilon_1)\langle v,D_\eta \rangle \bigg(\big(1-\chi(\eta/\epsilon_2)\big)\cdot\beta(x,\xi,\eta)\bigg)d\xi d\eta\bigg|\\
&=& |v|^{-2} \bigg |\int e^{-i(u\cdot\xi+v\cdot\eta)}\cdot\chi(\xi/\epsilon_1)\bigg(\epsilon_2^{-1}(\langle v,D_\eta \rangle \chi )(\eta/\epsilon_2)\beta(x,\xi,\eta)  \\
& &\qquad \qquad \qquad \qquad - \big(1-\chi(\eta/\epsilon_2)\big)\cdot \langle v,D_\eta \rangle \beta(x,\xi,\eta)\bigg)d\xi d\eta\bigg|\\
&\leq & |v|^{-1} \epsilon_1 (C_1-C_2\log \epsilon_2) \\
\end{eqnarray*}
Similarly, we can obtain that $I_3\leq |u|^{-1}\epsilon_2(C_1'-C_2'\log \epsilon_1)$ and
\begin{eqnarray*}
I_4&=&\bigg |\int e^{-i(u\cdot\xi+v\cdot\eta)}\cdot\big(1-\chi(\xi/\epsilon_2)\big)\big(1-\chi(\eta/\epsilon_2)\big)\cdot\beta(x,\xi,\eta)d\xi d\eta\bigg|\\
&=& |u|^{-2}|v|^{-2} \bigg |\int e^{-i(u\cdot\xi+v\cdot\eta)}\langle u,D_\xi \rangle \langle v,D_\eta \rangle \bigg(\big(1-\chi(\xi/\epsilon_2)\big)\big(1-\chi(\eta/\epsilon_2)\big)\beta(x,\xi,\eta)\bigg)d\xi d\eta\bigg|\\
&=& |u|^{-2}|v|^{-2}\bigg|\int e^{-i(u\cdot\xi+v\cdot\eta)}\cdot  \bigg(-\epsilon_1^{-1}\epsilon_2^{-1}(\langle u,D_\xi \rangle \chi) (\xi/\epsilon_1) (\langle v,D_\eta \rangle \chi) (\eta/\epsilon_2)\beta(x,\xi,\eta)  \\
& &+\epsilon_1^{-1} (\langle u,D_\xi \rangle \chi )(\xi/\epsilon_1) \big(1-\chi(\eta/\epsilon_2)\big) \langle v,D_\eta \rangle \beta(x,\xi,\eta) \\
 & &+\epsilon_2^{-1} \big(1-\chi(\xi/\epsilon_1)\big) (\langle v,D_\eta \rangle \chi )(\eta/\epsilon_2) \langle u,D_\xi\rangle \beta(x,\xi,\eta)  \\
& &- \big(1-\chi(\xi/\epsilon_1)\big) \big(1-\chi(\eta/\epsilon_2)\big) \langle u,D_\xi \rangle \langle v,D_\eta \rangle  \beta(x,\xi,\eta)\bigg) d\xi d\eta\bigg| \\
&\leq & |u|^{-1}|v|^{-1}(C_3-C_4 \log \epsilon_2 -C_5 \log \epsilon_1 + C_6 \log \epsilon_1 \log \epsilon_2 ).
\end{eqnarray*}

\vskip0.5cm

Thus $|B(x,u,v)|$ has a upper bound
\begin{eqnarray*}
 |u|^n |v|^n |B(x,u,v)|&\leq &C_0 \epsilon_1 \epsilon_2+  |v|^{-1} \epsilon_1 (C_1-C_2\log \epsilon_2)+ |u|^{-1}\epsilon_2(C_1'-C_2'\log \epsilon_1)\\
  & & \quad + |u|^{-1}|v|^{-1}(C_3-C_4 \log \epsilon_2 -C_5 \log \epsilon_1+   C_6 \log \epsilon_1 \log \epsilon_2 ),
\end{eqnarray*}
where $C_i(1\leq i\leq 6)$ are some positive constants.  Taking $\epsilon_1=|u|^{-1}\in(0,1],\epsilon_2=|v|^{-1}\in(0,1]$, we get
\begin{eqnarray*}
|u|^n|v|^n|B(x,u,v)|&\leq& |u|^{-1}|v|^{-1}\bigg(C+C\big|\log |v|\big|+C\big|\log|u|\big|+C\big|\log|u|\log |v|\big|\bigg)
\\
&\lesssim& |u|^{-1}|v|^{-1}\bigg(1+log|u|\bigg)\bigg(1+log|v|\bigg)\lesssim|u|^{-\mu}|v|^{-\mu},\ \forall 0\leq\mu<1\\
\Rightarrow&&|B(x,u,v)|\leq|u|^{-n-\mu}|v|^{-n-\mu}
\end{eqnarray*}
So for all $0\leq\mu<1$, we have\\
$\sup\limits_{x,u,v\in\mathbb{R}^n}(1+|u|)^{n+\mu}\cdot(1+|v|)^{n+\mu}\bigg|\int e^{-i(u\cdot\xi+v\cdot\eta)}a(x,\xi,\eta)d\xi d\eta\bigg|<\infty$
\end{proof}

\vskip0.3cm
The following lemma allows us to change variables for the non-smooth substitution.
\begin{Lemma}[\cite{FS}]\label{le-FS}
  Let $U$ be a measurable set in $\mathbb{R}^n$ and let $t:U\to \mathbb{R}^n$ be a bounded measurable map satisfying
  $$|t(x)-t(y)|\geq C |x-y|$$
  for almost every  $x,y \in U$. Then there exists a function $J_t\in L^\infty(\mathbb{R}^n)$ supported in $t(u)$ such that the substitution formula
  $$\int_U u\circ t(x) dx = \int u(z)J_t(z)dz$$
  holds for all $u\in L^1(\mathbb{R}^n)$ and the Jacobian $J_t$ satisfies the estimate $\|J_t\|_{L^\infty}\leq \frac{2\sqrt n}{c}$.
\end{Lemma}
\begin{corollary}
\label{changev}
  Let $t: \mathbb{R}^n\to \mathbb{R}^n$ be a map satisfying the assumptions in the previous lemma with $U=\mathbb{R}^n$, then $u\mapsto u\circ t$ is a bounded map on $L^p$ for $p\in[1,\infty]$.
\end{corollary}
\setcounter{equation}{0}

\section{\bf The  Seeger-Sogge-Stein decomposition for the bi-parameter FIOs}\label{sec-decomp}

 We begin with recalling the following
Seeger-Sogge-Stein decomposition initiated in their work of one-parameter FIOs in $\mathbb{R}^{n}$ \cite{SSS}.

\medskip

For any $s>0$, choose a set of unit vectors $\{\xi_s^\mu\}_\mu$ such that for all $\mu,\mu'$
$$\big|\xi_s^\mu-\xi_s^{\mu'}\big|\approx2^{-s/2}.$$

We want the union of the balls of radii $2^{-s/2}$ centered at  $\xi_s^\mu$ to cover the unit sphere in $\mathbb{R}^{n}$. Let $\Gamma_s^\mu$ denote the cone in the $\xi$ space whose central direction is $\xi_s^\mu$, i.e
$$\Gamma_s^\mu=\bigg\{\xi:\bigg|\frac\xi{|\xi|}-\xi_s^\mu\bigg|\leq2\cdot2^{-s/2}\bigg\}$$

 We can then select a set of unit vectors $\{\xi_s^\mu\}_\mu$ of cardinality $c_n\cdot2^{s(n-1)/2}$ that meet all the above conditions.
Let $\{\chi_s^\mu\}$ be a partition of
unity on the unit sphere subordinate to this covering which satisfies the following conditions:
\begin{itemize}
\item[$(1)$]Each  $\chi_s^\mu$ is homogeneous of degree 0 in $\xi$ and supported in $\Gamma_s^\mu$ with $\sum\limits_\mu\chi_s^\mu(\xi)=1, \ \forall s, \forall\xi\neq0;$\\
\item[$(2)$] $\big|\partial_\xi^\alpha\chi_s^\mu(\xi)\big|\leq C_\alpha\cdot2^{\frac{|\alpha|s}2}\cdot|\xi|^{-|\alpha|}$, with the improvement $\big|\partial_{\xi_1}^N\chi_s^\mu(\xi)\big|\leq C_N\cdot|\xi|^{-N}$ for  $N\geq1$, if one chooses the axis in
$\xi$ plane such that $\xi_1$ is in the direction of $\xi^\mu_s$.\\
\end{itemize}
\vskip0.2cm
We can then  decompose the bi-parameter FIO $T$ in \eqref{opdef} as follows: Taking the Littlewood-Paley partition of unity in  $\xi\in \mathbb{R}^n$ $$1=\big(\Psi_0(\xi)+\sum\limits_{j=1}^\infty \Psi_j(\xi)\big),$$
where $\supp\,\Psi_0\subseteq\{\xi:|\xi|\leq2\},\supp\,\Psi\subseteq A_\xi :=\{\xi:\frac12\leq|\xi|\leq2\},\Psi_j(\xi)=\Psi(2^{-j}\xi).$

\medskip

By doing the Littlewood-Paley decomposition simultaneously in both $\xi, \eta\in \mathbb{R}^n$ variables, then we can write
\begin{eqnarray*}
  T(f)(x)&=& \int_{\mathbb{R}^{2n}}\big(\Psi_0(\xi)+\sum\limits_{j=1}^\infty \Psi_j(\xi)\big) \big(\Psi_0(\eta)+\sum\limits_{k=1}^\infty \Psi_k(\eta)\big)e^{i\varphi(x,\xi,\eta)}a(x,\xi,\eta)\widehat{f}(\xi,\eta)d\xi d\eta \\
  &=& \int_{\mathbb{R}^{2n}}\big(\Psi_0(\xi)\Psi_0(\eta)+\Psi_0(\eta)\sum\limits_{j=1}^\infty \Psi_j(\xi) + \Psi_0(\xi)\sum\limits_{k=1}^\infty \Psi_k(\eta) \\
  & & \ + \sum\limits_{j=1}^\infty \Psi_j(\xi) \sum\limits_{k=1}^\infty \Psi_k(\eta)\big)e^{i\varphi(x,\xi,\eta)}\cdot a(x,\xi,\eta)\cdot\widehat{f}(\xi,\eta)d\xi d\eta \\
  &:=& T_{00}(f)+\sum\limits_{j=1}^\infty T_{j0}(f)+\sum\limits_{k=1}^\infty T_{0k}(f)+\sum\limits_{j,k=1}^\infty T_{jk}(f)
\end{eqnarray*}
For $j,k\geq 1$, by using the  Seeger-Sogge-Stein  decomposition write $T_{jk}$ as follows
\begin{eqnarray*}
 & & T_{jk}(f)(x) \\
 &= & 2^{(j+k)n}\int_{\mathbb{R}^{2n}} e^{i (2^j\varphi_1(x_1,\xi)+2^k\varphi_2(x_2,\eta))}\Psi(\xi)\Psi(\eta) a(x,2^j\xi,2^k\eta)\cdot\widehat{f}(2^j\xi,2^k\eta)d\xi d\eta \\
  &=& 2^{(j+k)n}\int_{\mathbb{R}^{2n}} (\sum_\mu \chi_j^\mu(\xi))(\sum_\nu \chi_k^\nu(\eta))e^{i (2^j\varphi_1(x_1,\xi)+2^k\varphi_2(x_2,\eta))}\Psi(\xi)\Psi(\eta) a(x,2^j\xi,2^k\eta)\widehat{f}(2^j\xi,2^k\eta)d\xi d\eta \\
  &:=&\sum_{\mu\nu} T^{\mu,\nu}_{jk}(f)(x).
\end{eqnarray*}
We can choose the coordinates on $\mathbb{R}^{n}= \mathbb{R}\xi^\mu \oplus \xi^{\mu \bot}= \mathbb{R}\eta^\nu \oplus \eta^{\nu \bot}$ in the way $$\xi=\xi_1\xi^\mu+\xi', \qquad \text{and} \qquad \eta=\eta_1\eta^\nu+\eta',$$
then the kernel of the operator $T^{\mu,\nu}_{jk}$ is given by
\begin{eqnarray*}
  & &K^{\mu,\nu}_{jk}(x,y)\\
  &=& 2^{(j+k)n} \int_{\mathbb{R}^{2n}} \Psi(\xi) \Psi(\eta)\chi^\mu_j(\xi)  \chi^\nu_k(\eta) e^{i 2^j(\varphi_1(x_1,\xi)-\langle y_1,\xi \rangle)}e^{i 2^k(\varphi_2(x_2,\eta)-\langle y_2, \eta \rangle )} a(x,2^j\xi,2^k\eta)d\xi d\eta \\
  &=& 2^{(j+k)n}\int_{\mathbb{R}^{2n}}  e^{i ( 2^j \langle \nabla_\xi \varphi_1(x_1,\xi^\mu)- y_1,\xi \rangle)} e^{i ( 2^k \langle \nabla_\eta \varphi_2(x_2,\eta^\nu)-y_2, \eta \rangle )} b^{\mu,\nu}_{jk}(x,\xi,\eta)d\xi d\eta,
\end{eqnarray*}
with
\begin{eqnarray}
  \label{kernel-kj}
  b^{\mu,\nu}_{jk}(x,\xi,\eta)&=& \Psi(\xi) \Psi(\eta)\chi^\mu_j(\xi)  \chi^\nu_k(\eta) e^{i ( 2^j \langle \nabla_\xi \varphi_1(x_1,\xi)-\nabla_\xi \varphi_1(x_1,\xi^\mu),\xi \rangle)} \\
  & & \ \ \cdot e^{i ( 2^k \langle \nabla_\eta \varphi_2(x_2,\eta)-\nabla_\eta \varphi_2(x_2,\eta^\nu), \eta \rangle )} a(x,2^j\xi,2^k\eta).\nonumber
\end{eqnarray}
Note that
\begin{eqnarray}
\label{estSSS}
  &&\sup_{\xi,\eta} \|\partial^{\alpha}_\xi \partial^{\beta}_\eta \big( \Psi(\xi) \Psi(\eta) a(x,2^j\xi,2^k\eta) \big) \|_{L^\infty}\\
  &&\leq C_{\alpha,\beta} 2^{j\big(\mi+|\alpha|(1-\rhi)\big)}  2^{k\big(\mii+|\beta|(1-\rhii)\big)}\nonumber
\end{eqnarray}
for all multi-indices $\alpha,\beta$.
\vskip0.2cm

The following lemma gives us an estimate of the kernel.
\begin{Lemma}
\label{le-kernel}
Let $a\in L^\infty BS^m_\rho$ and $\varphi(x,\xi,\eta)$ be defined as in $\eqref{phasedef}$. Then the  symbol $b^{\mu,\nu}_{jk}(x,\xi,\eta)$ satisfies the estimates

\begin{eqnarray*}
  \sup_{\xi,\eta} \|\partial^{\alpha}_\xi \partial^{\beta}_\eta b^{\mu,\nu}_{jk}(\cdot, \xi,\eta)\|_{L^\infty} \leq C_{\alpha,\beta} 2^{j(\mi+|\alpha|(1-\rhi)+\frac{|\alpha'|}{2})}  2^{k(\mii+|\beta|(1-\rhii)+\frac{|\beta'|}{2})}.
\end{eqnarray*}
where $\alpha=(\alpha_1,\alpha_2,...,\alpha_n)\doteq(\alpha_1,\alpha'),\ \beta=(\beta_1,\beta_2,...,\beta_n)\doteq(\beta_1,\beta')$.

\end{Lemma}

\begin{proof}
 This lemma is a direct result from \cite{FS}. By \cite{FS}, we have:
  $$\sup_{\xi\in A_\xi}|\partial^\alpha_\xi \chi^\mu_j (\xi)|\leq C_\alpha 2^{\frac{j|\alpha'|}{2}},\quad \sup_{\eta\in A_\eta}|\partial^\beta_\eta \chi^\nu_k (\eta)|\leq C_\beta 2^{\frac{k|\beta'|}{2}},$$

  \begin{eqnarray*}
  \sup_{\xi\in A_\xi \cap \Gamma^\mu_j}\|\partial^\alpha_\xi \big( e^{i ( 2^j \langle \nabla_\xi \varphi_1(x_1,\xi^\mu)- \nabla_\xi \varphi_1(x_1,\xi),\xi \rangle)}\big)\|_{L^\infty}\leq C_\alpha 2^{\frac{j|\alpha'|}{2}},
\end{eqnarray*}

\begin{eqnarray*}
  \sup_{\eta \in A_\eta \cap \Gamma^\nu_k}\|\partial^\beta_\eta \big( e^{i ( 2^k \langle \nabla_\eta \varphi_2(x_2,\eta^\nu)-\nabla_\eta \varphi_2(x_2,\eta), \eta \rangle )} \big)\|_{L^\infty}\leq C_\beta 2^{\frac{k|\beta'|}{2}}.
\end{eqnarray*}
Together with $\eqref{estSSS}$, we can get the desired estimate.
\end{proof}

\setcounter{equation}{0}

\section{Proof  of Theorem $\ref{th-a}$}
\label{sec-thpq}

The proof is divided into several steps.

\medskip

\begin{proof}

\vskip0.5cm

Step I: Since $\supp_\xi\, a,\supp_\eta \, a$ are compact,  then there exist closed cubes $Q_1,Q_2$ of side length $L_1,L_2$ such that
$\supp_\xi \,a\subseteq \text{Int}(Q_1)$, $\supp_\eta \,a\subseteq \text{Int}(Q_2)$. Then we can extend $a(x,\cdot,\cdot)|_{Q_1\times Q_2}$ periodically into $\widetilde{a}(x,\cdot,\cdot)\in C^\infty(\mathbb{R}_\xi^n\times\mathbb{R}_\eta^n)$.

\medskip

We can choose $\varsigma_1,\varsigma_2\in C_0^\infty$ with $\supp\,\varsigma_1\subseteq Q_1,\supp\,\varsigma_2\subseteq Q_2$, and $\varsigma_1=1$ on $\supp_\xi\, a$,\ \ $\varsigma_2=1$ on $\supp_\eta \, a$ such that $a(x,\xi,\eta)=\widetilde{a}(x,\xi,\eta)\cdot\varsigma_1(\xi)\cdot\varsigma_2(\eta)$. We then expand $\widetilde{a}(x,\xi,\eta)$ in a Fourier series:
$$\widetilde{a}(x,\xi,\eta)=\sum\limits_{\gfz{k\in\mathbb{Z}^n\times\mathbb{Z}^n}{k=(k_1,k_2)}}a_k(x)\cdot e^{i\frac{2\pi}{L_1}(k_1\cdot\xi)} e^{i\frac{2\pi}{L_2}(k_2\cdot\eta)}$$
where
\begin{eqnarray*}
a_k(x)&=&\frac{1}{L_1^n} \frac{1}{L_2^n}\int_{Q_1\times Q_2}\widetilde{a}(x,\xi,\eta) e^{-i\frac{2\pi}{L_1}(k_1\cdot\xi)} e^{-i\frac{2\pi}{L_2}(k_2\cdot\eta)} d\xi d\eta\\
&=&\frac{1}{L_1^n} \frac{1}{L_2^n}\int_{\mathbb{R}^{2n}}a(x,\xi,\eta) e^{-i\frac{2\pi}{L_1}(k_1\cdot\xi)} e^{-i\frac{2\pi}{L_2}(k_2\cdot\eta)} d\xi d\eta
\end{eqnarray*}

\medskip

Setting $f_k(x)=f_{k_1,k_2}(x)=f(x_1+\frac{2\pi}{L_1}k_1,x_2+\frac{2\pi}{L_2}k_2 )$, we have
\begin{eqnarray*}
 T(f)(x)&=&\int_{\mathbb{R}^{2n}}e^{i\varphi(x,\xi,\eta)}\cdot a(x,\xi,\eta)\cdot\widehat{f}(\xi,\eta)d\xi d\eta\\
&=&\int_{\mathbb{R}^{2n}}e^{i\varphi(x,\xi,\eta)}\cdot \widetilde{a}(x,\xi,\eta)\cdot\varsigma_1(\xi)\cdot\varsigma_2(\eta)\cdot\widehat{f}(\xi,\eta)d\xi d\eta\\
&=&\sum\limits_{\gfz{k\in\mathbb{Z}^n\times\mathbb{Z}^n}{k=(k_1,k_2)}}a_k(x)\int_{\mathbb{R}^{2n}}e^{i\varphi(x,\xi,\eta)}\cdot e^{i\frac{2\pi}{L_1}(k_1\cdot\xi)} e^{i\frac{2\pi}{L_2}(k_2\cdot\eta)}\cdot\varsigma_1(\xi)\cdot\varsigma_2(\eta)\cdot\widehat{f}(\xi,\eta)d\xi d\eta\\
&=&\sum\limits_{\gfz{k\in\mathbb{Z}^n\times\mathbb{Z}^n}{k=(k_1,k_2)}}a_k(x)\int_{\mathbb{R}^{2n}}e^{i\varphi(x,\xi,\eta)}\cdot \varsigma_1(\xi)\cdot\varsigma_2(\eta)\cdot \widehat{f_{k_1,k_2}}(\xi,\eta)d\xi d\eta\\
&=&\sum\limits_{k\in\mathbb{Z}^n\times\mathbb{Z}^n}a_k(x)T_{\varsigma_1,\varsigma_2}(f_{k})(x),
\end{eqnarray*}
where $T_{\varsigma_1,\varsigma_2}(f_{k})(x):=\int_{\mathbb{R}^{2n}}e^{i\varphi(x,\xi,\eta)}\cdot \varsigma_1(\xi)\cdot\varsigma_2(\eta)\cdot \widehat{f_{k}}(\xi,\eta)d\xi d\eta$.

\vskip0.5cm
Step II: We will prove that $T_{\varsigma_1,\varsigma_2}$ is bounded from $L^q$ to $L^q$. By Lemma \ref{le-a}, for some $\zeta_1,\zeta_2\in \mathbb{S}^{n-1}$.
\begin{eqnarray*}
 &&T_{\varsigma_1,\varsigma_2}(f)(x)=\int_{\mathbb{R}^{2n}} e^{i(\psi_1(x_1,\xi)+\langle\nabla_{\xi}\varphi_1(x_1,\zeta_1),\xi\rangle)}e^{i(\psi_2(x_2,\eta)+\langle\nabla_{\eta}\varphi_2(x_2,\zeta_2),\eta\rangle)} \widehat{f}(\xi,\eta)d\xi d\eta\\
 &&=\int_{\mathbb{R}^{2n}}\bigg\{\int_{\mathbb{R}^{2n}} e^{i\psi_1(x_1,\xi)}  e^{i\psi_2(x_2,\eta)}e^{i\langle\nabla_{\xi}\varphi(x_1,\zeta_1)-y_1,\xi\rangle}  e^{i\langle\nabla_{\eta}\varphi(x_2,\zeta_2)-y_2,\eta\rangle} \varsigma_1(\xi)\varsigma_2(\eta)d\xi d\eta\bigg\}f(y)dy\\
&& :=\int_{\mathbb{R}^{2n}}k(x,y)f(y)dy
\end{eqnarray*}
where $k(x,y)=\int_{\mathbb{R}^{2n}} e^{i\psi_1(x_1,\xi)}  e^{i\psi_2(x_2,\eta)}\cdot e^{i\langle\nabla_{\xi}\varphi(x_1,\zeta_1)-y_1,\xi\rangle}  e^{i\langle\nabla_{\eta}\varphi(x_2,\zeta_2)-y_2,\eta\rangle} \varsigma_1(\xi)\varsigma_2(\eta)d\xi d\eta$. Lemma \ref{le-b} implies that for any $s\in [0,1)$
\begin{eqnarray*}
|k(x,y)| \lesssim  \langle \nabla_{\xi}\varphi(x_1,\zeta_1)-y_1\rangle^{-n-s} \langle  \nabla_{\eta}\varphi(x_2,\zeta_2)-y_2\rangle^{-n-s}.
\end{eqnarray*}
Then we have $$\sup_ x \int |k(x,y)| dy <\infty.$$
Then we estimate  $\int |k(x,y)| dx,$ where we can do the change of variables $z_1=\nabla_{\xi}\varphi(x_1,\zeta_1), z_2=\nabla_{\eta}\varphi(x_2,\zeta_2)$, based on the strong non-degeneracy condition \eqref{SND} and lemma \ref{le-FS}. Moreover, by Schwartz's global inverse function theorem, the corresponding Jacobian $J(z)$ satisfies $|\det J(z)|\lesssim 1/c$. Then for all $s\in [0,1)$
\begin{eqnarray*}
  \sup_ y \int |k(x,y)| dx = \sup_y  \int \langle z_1-y_1\rangle^{-n-s} \langle z_2-y_2\rangle^{-n-s} dz<\infty
\end{eqnarray*}
So we can conclude $T$ is bounded on $L^1$ and $L^\infty$ as well. The interpolation implies that $T_{\varsigma_1,\varsigma_2}$ is bounded from $L^q$ to $L^q$.

\vskip0.5cm

Step III: We first estimate $a_k(x)$. By doing the integration by parts sufficiently many times for $a_k(x)=\int_{\mathbb{R}^{2n}} {a}(x,\xi,\eta)e^{-i\frac{2\pi}{L_1}(k_1\cdot\xi)} e^{-i \frac{2\pi}{L_1}(k_2\cdot\eta)}d\xi d\eta.$ We can choose large enough $N$ such that
\begin{eqnarray*}
\|a_k\|_{L^{p}}&\approx& \bigg\| \int_{\mathbb{R}^{2n}}{a}(x,\xi,\eta)e^{-i(k_1\cdot\xi+k_2\cdot\eta)}d\xi d\eta\bigg\|_{L^{p}}\\
&\lesssim& \frac{1}{(1+|k_1|)^{N}} \frac{1}{(1+|k_2|)^{N}}\bigg\| \int_{\mathbb{R}^{2n}}\big(\partial^\alpha_\xi \partial^\beta_\eta {a}\big)(x,\xi,\eta) e^{-i(k_1\cdot\xi+k_2\cdot\eta)}d\xi d\eta\bigg\|_{L^{p}}\\
&\lesssim &   \frac{1}{(1+|k_1|)^{N}} \frac{1}{(1+|k_2|)^{N}}\int_{Q_1\times Q_2}\|\big(\partial^\alpha_\xi \partial^\beta_\eta {a}\big)(x,\xi,\eta)\|_{L^{p}}d\xi d\eta \\
&\lesssim&  \frac{1}{(1+|k_1|)^{N}} \frac{1}{(1+|k_2|)^{N}}.
\end{eqnarray*}
Therefore, for $1\leq r\leq \infty$,
\begin{eqnarray*}
&&\|T(f)\|_{L^r}=\bigg\|\sum\limits_{k\in\mathbb{Z}^n\times\mathbb{Z}^n}a_k(x)T_{\varsigma_1,\varsigma_2}(f_k)\bigg\|_{L^r}\leq\sum\limits_{k\in\mathbb{Z}^n\times\mathbb{Z}^n}\big\|a_k(x)T_{\varsigma_1,\varsigma_2}(f_k)\big\|_{L^r}\\
&&\leq\sum\limits_{k\in\mathbb{Z}^n\times\mathbb{Z}^n}\|a_k\|_{L^{p}}\big\|T_{\varsigma_1,\varsigma_2}(f_k)\big\|_{L^q}\lesssim \sum\limits_{k\in\mathbb{Z}^n\times\mathbb{Z}^n}\|a_k\|_{L^{p}}\|f_k\|_{L^{q}}\\
&&\lesssim\sum\limits_{k\in\mathbb{Z}^n\times\mathbb{Z}^n}(1+|k_1|)^{-|N|}\cdot(1+|k_2|)^{-|N|}\|f_k\|_{L^{q}}\lesssim\|f\|_{L^{q}}
\end{eqnarray*}
When $0<r<1$,
\begin{eqnarray*}
&&\|T(f)\|_{L^r}^r= \sum\limits_{k\in\mathbb{Z}^n\times\mathbb{Z}^n}\bigg\|a_k(x)T_{\varsigma_1,\varsigma_2}(f_k)\bigg\|_{L^r}^r \leq\sum\limits_{k\in\mathbb{Z}^n\times\mathbb{Z}^n}\|a_k\|_{L^{p}}^r\big\|T_{\varsigma_1,\varsigma_2}(f_k)\big\|^r_{L^q} \\
&&\lesssim \sum\limits_{k\in\mathbb{Z}^n\times\mathbb{Z}^n}\|a_k\|_{L^{p}}^r\|f_k\|_{L^{q}}^r \lesssim\sum\limits_{k\in\mathbb{Z}^n\times\mathbb{Z}^n}(1+|k_1|)^{-r|N|}\cdot(1+|k_2|)^{-r|N|}\|f_k\|_{L^{q}}^r\\
&&\lesssim\|f\|_{L^{q}}^r
\end{eqnarray*}
This completes the proof.
\end{proof}

\setcounter{equation}{0}
\section{\bf $L^p$ boundedness for bi-parameter FIOs with non-smooth phase in $x$: Proof of Theorem 1.4}
In the following proofs, we will take advantage of the decomposition introduced in Section \ref{sec-decomp}.
\label{th-lp}
\subsection{Proof of Theorem \ref{th-local}.}
By Lemma \ref{le-a}, we can write
\begin{eqnarray*}
T_{00}f(x)=\int_{\mathbb{R}^{2n}}e^{i(\psi_1(x_1,\xi)+\langle\nabla_{\xi}\varphi_1(x_1,\zeta_1),\xi\rangle)} e^{i(\psi_2(x_2,\eta)+\langle\nabla_{\eta}\varphi_2(x_2,\zeta_2),\eta\rangle)}\Phi^1_0(\xi)\Phi^2_0(\eta) a(x,\xi,\eta)\widehat{f}(\xi,\eta)d\xi d\eta,
\end{eqnarray*}
for some $\zeta_1,\zeta_2 \in \mathcal{S}^{n-1}$ and $\psi_1,\psi_2 \in L^\infty \Phi^1$. The corresponding kernel is given by
\begin{eqnarray*}
  K_{00}(x,y)= \int_{\mathbb{R}^{2n}}e^{i\psi_1(x_1,\xi)}e^{\langle\nabla_{\xi}\varphi_1(x_1,\zeta_1)-y_1,\xi\rangle} e^{i\psi_2(x_2,\eta)}e^{\langle\nabla_{\eta}\varphi_2(x_2,\zeta_2)-y_2,\eta\rangle}\Phi^1_0(\xi)\Phi^2_0(\eta) a(x,\xi,\eta)d\xi d\eta
\end{eqnarray*}

Recall for $|\alpha|,|\beta|\geq 1$, for all $x=(x_1,x_2)$,
$$\sup_{|\xi|\neq 0} |\xi|^{-1+|\alpha|} |\partial^\alpha_\xi \psi_1(x_1,\xi)|<\infty,\  \sup_{|\eta|\neq 0} |\eta|^{-1+|\beta|} |\partial^\beta_\eta \psi_2(x_2,\eta)|<\infty.$$
Then for $b(x,\xi,\eta)= e^{i\psi_1(x_1,\xi)} e^{i\psi_2(x_2,\eta)} \Phi^1_0(\xi)\Phi^2_0(\eta)a(x,\xi,\eta) $, the following holds uniformly in $x$:
$$\sup_{|\xi|,|\eta|\neq 0}|\xi|^{-1+|\alpha|} |\eta|^{-1+|\beta|} |\partial^\alpha_\xi\partial^\beta_\eta b(x,\xi,\eta)|<\infty \quad \text{for}\; |\alpha|,|\beta|\geq 1 .$$
By Lemma \ref{le-b}, for all $\mu\in [0,1)$,
$$ |K_{00}(x,y)|\lesssim \langle\nabla_{\xi}\varphi_1(x_1,\zeta_1)-y_1\rangle^{-n-\mu} \langle\nabla_{\eta}\varphi_2(x_2,\zeta_2)-y_2\rangle^{-n-\mu}.$$
Thus, we have $$\sup_x \int |K_{00}(x,y)|dy <\infty,$$
with the non-smooth change of variables and rough non-degeneracy condition,
\begin{eqnarray*}
\sup_y \int |K_{00}(x,y)|dx &\lesssim& \sup_y  \int \langle\nabla_{\xi}\varphi_1(x_1,\zeta_1)-y_1\rangle^{-n-\mu} \langle\nabla_{\eta}\varphi_2(x_2,\zeta_2),\eta\rangle^{-n-\mu} dx \\
&\lesssim& \int \langle z \rangle^{-n-\mu} dz <\infty.
\end{eqnarray*}

So we have $T_{00}(f)$ is bounded on $L^\infty,L^1$ respectively,  and therefore bounded on $L^p$ for $p\in [1,\infty]$.

\subsection{Proof of Theorem \ref{th-b}.}
\begin{proof}
Theorem \ref{th-local} implies the desired estimate for $T_{00}$. For other cases, first consider for $j,k\geq 1$

$$T_{jk}=\sum_{\mu,\nu} T^{\mu,\nu}_{jk}.$$
Let us consider the following differetial operators
$$L^N=(1-\partial^2_{\xi_1}-2^{-j}\partial^2_{\xi'})^{N}(1-\partial^2_{\eta_1}-2^{-k}\partial^2_{\eta'})^{N}, \ \ \ N\in \mathbb{N}_+.$$
By Lemma $\ref{le-kernel}$, we have
\begin{eqnarray*}
\sup_{\xi,\eta} \|L^N(b^{\mu,\nu}_{jk}(x,\xi,\eta))\|_{L^\infty}\leq C2^{j(m^{(1)}+2 N(1-\rho^{(1)}))}\cdot2^{k(m^{(2)}+2 N(1-\rho^{(2)}))}.
\end{eqnarray*}

Now for any integer $m\geq 1$, we define the function $$g_m(z)=2^{2m}z_1^2+2^m|z'|^2 ,\quad z\in \mathbb{R}^n,$$ then from the  integration by parts,
\begin{eqnarray*}
 & & |K^{\mu,\nu}_{jk}(x,y)|\\
  &= &2^{(j+k)n} \int_{A_\xi\cap \Gamma^\mu_j} \int_{A_\eta\cap \Gamma^\nu_k}  e^{i ( 2^j \langle \nabla_\xi \varphi_1(x_1,\xi^\mu)- y_1,\xi \rangle)} e^{i ( 2^k \langle \nabla_\eta \varphi_2(x_2,\eta^\nu)-y_2, \eta \rangle )} b^{\mu,\nu}_{jk}(x,\xi,\eta)d\xi d\eta \\
  &\leq& 2^{(j+k)n} \big(1+g_j(\nabla_\xi \varphi_1(x_1,\xi^\mu)- y_1)\big)^{-N} \big(1+g_k(\nabla_\eta \varphi_2(x_2,\eta^\nu)-y_2)\big)^{-N} \int |L^N(b^{\mu,\nu}_{jk}(x,\xi,\eta))| d\xi d\eta \\
  &\leq& C_{l_1,l_2} 2^{j(\mi+{n+1\over 2}+2N(1-\rhi))}  2^{k(\mii+{n+1\over 2}+2N(1-\rhii))} \\
   & & \quad \cdot \big(1+g_j(\nabla_\xi \varphi_1(x_1,\xi^\mu)- y_1)\big)^{-N} \big(1+g_k(\nabla_\eta \varphi_2(x_2,\eta^\nu)-y_2)\big)^{-N}
\end{eqnarray*}
holds for all non-negative integers $N$, since $|A_\xi\cap \Gamma^\mu_j|\approx 2^{-j(n-1)\over 2}$ and $|A_\eta\cap \Gamma^\nu_k|\approx 2^{-k(n-1)\over 2}$.

For any positive number $M$, say ${M\over 2}=N+\theta$, where the integer part $N=[{M\over 2}]$ and $\theta\in [0,1)$, by interpolation
\begin{eqnarray*}
  |K^{\mu,\nu}_{jk}(x,y)|&=& |K^{\mu,\nu}_{jk}(x,y)|^\theta  |K^{\mu,\nu}_{jk}(x,y)|^{(1-\theta)} \\
  &\leq & C_{N,N}^{1-\theta} C_{N+1,N+1}^\theta 2^{j(\mi+{n+1\over 2}+M(1-\rhi))}  2^{k(\mii+{n+1\over 2}+M(1-\rhii))}\\
  & & \quad \cdot \big(1+g_j(\nabla_\xi \varphi_1(x_1,\xi^\mu)- y_1)\big)^{-{M\over 2}} \big(1+g_k(\nabla_\eta \varphi_2(x_2,\eta^\nu)-y_2)\big)^{-{M\over 2}}.
\end{eqnarray*}
Thus, for any $M>n$,
$$\sup_{x} \int  |K^{\mu,\nu}_{jk}(x,y)| dy \lesssim 2^{j(\mi+M(1-\rhi))}2^{k(\mii+M(1-\rhii))}.$$
Taking advantage of the rough non-degeneracy assumption  and Corollary $\ref{changev}$, we have
\begin{eqnarray*}
  & &\sup_{y} \int  |K^{\mu,\nu}_{jk}(x,y)| dx \leq C_{N,N}^{1-\theta} C_{N+1,N+1}^\theta 2^{j(\mi+{n+1\over 2}+M(1-\rhi))}  2^{k(\mii+{n+1\over 2}+M(1-\rhii))}\\
  & & \qquad  \cdot  \sup_{y} \int \big(1+g_j(\nabla_\xi \varphi_1(x_1,\xi^\mu)- y_1)\big)^{-{M\over 2}} \big(1+g_k(\nabla_\eta \varphi_2(x_2,\eta^\nu)-y_2)\big)^{-{M\over 2}}dx \\
  &\lesssim& 2^{j(\mi+M(1-\rhi))}2^{k(\mii+M(1-\rhii))}.
\end{eqnarray*}
Thus,
\begin{eqnarray*}
  \|T_{jk} f\|_{L^1}&\leq& \sum_{\mu,\nu} \|T_{jk}^{\mu,\nu} f\|_{L^1}
  \lesssim 2^{j(\mi+{n-1\over 2}+M(1-\rhi))}2^{k(\mii+{n-1\over 2}+M(1-\rhii))}\|f\|_{L^1},
\end{eqnarray*}
is summable for $j,k$ provided $\mi<-{n-1\over 2}-M(1-\rhi), \mii<-{n-1\over 2}-M(1-\rhii)$ and $M>n$.

\vskip0.2cm

Now we are ready to consider the case $T_{0k}$, and $T_{j0}$ can be treated similarly.
\begin{eqnarray*}
  \sum_{k=1}^\infty T_{0k}(f)(x)&=& \sum_{k=1}^\infty  2^{kn} \int e^{i(\varphi_1(x_1,\xi)+ 2^k \varphi_2(x_2,\eta))}\Psi_0(\xi) \Psi(\eta) a(x,\xi,2^k\eta) \hat f(\xi,2^k\eta) d\xi d\eta \\
  &=& \sum_{k=1}^\infty 2^{kn} \int e^{i(\varphi_1(x_1,\xi)+ 2^k \varphi_2(x_2,\eta))}(\sum_\nu \chi^\nu_k)\Psi_0(\xi) \Psi(\eta) a(x,\xi,2^k\eta) \hat f(\xi,2^k\eta) d\xi d\eta \\
  &=& \sum_{k=1}^\infty \sum_\nu T_{0k}^\nu(f)(x),
\end{eqnarray*}
where the kernel of $ T_{0k}^\nu$ is given by
\begin{eqnarray*}
  & &K^{\nu}_{0k}(x,y)\\
  &=& 2^{kn} \int_{\mathbb{R}^{2n}} \Psi_0(\xi) \Psi(\eta)  \chi^\nu_k(\eta) e^{i (\varphi_1(x_1,\xi)-\langle y_1,\xi \rangle)}e^{i 2^k(\varphi_2(x_2,\eta)-\langle y_2, \eta \rangle )} a(x,\xi,2^k\eta)d\xi d\eta \\
  &=& 2^{kn}\int_{\mathbb{R}^{2n}} e^{i   \langle \nabla_\xi \varphi_1(x_1,\zeta_1)- y_1,\xi \rangle} e^{i ( 2^k \langle \nabla_\eta \varphi_2(x_2,\eta^\nu)-y_2, \eta \rangle )} b^{\nu}_{0k}(x,\xi,\eta)d\xi d\eta,
\end{eqnarray*}
with
\begin{eqnarray*}
  b^{\nu}_{0k}(x,\xi,\eta)= \Psi_0(\xi) \Psi(\eta)  \chi^\nu_k(\eta)  e^{i \psi (x_1,\xi)} e^{i ( 2^k \langle \nabla_\eta \varphi_2(x_2,\eta)-\nabla_\eta \varphi_2(x_2,\eta^\nu), \eta \rangle )}a(x,\xi,2^k\eta).
\end{eqnarray*}
for some $\zeta_1\in \mathcal{S}{(\n)}$.
Note that as before for $|\alpha|\geq 1$ and all multi-indices $\beta$,
\begin{eqnarray}
\label{L1-k-in}
   \sup_{\eta \neq 0}\| \partial^\alpha_\xi \partial^\beta_\eta  b^{\nu}_{0k}(\cdot,\xi,\eta)\|_{L^\infty}\lesssim \frac{2^{k(\mii+|\beta|(1-\rhii)+\frac{|\beta'|}{2})}}{|\xi|^{|\alpha|-1}},
\end{eqnarray}
which follows from the estimates
\begin{eqnarray*}
 \sup_{\eta\neq0} \| \partial^\alpha_\xi \partial^\beta_\eta \big( \Psi_0(\xi)\Psi(\eta)e^{\psi_1(x_1,\xi)}a(x,\xi,2^k\eta) \big)\|_{L^\infty} \lesssim \frac{1}{|\xi|^{|\alpha|-1}} 2^{k(\mii+(1-\rhii))},\quad \text{for}\   \xi\neq 0,
\end{eqnarray*}

\begin{eqnarray*}
 \sup_{\eta\in A_\eta}|\partial^\alpha_\eta \chi^\nu (\eta)|\leq C_\beta 2^{\frac{k|\beta'|}{2}},\quad \sup_{\eta \in A_\eta \cap \Gamma^\nu_k}\|\partial^\beta_\eta \big( e^{i ( 2^k \langle \nabla_\eta \varphi_2(x_2,\eta^\nu)-\nabla_\eta \varphi_2(x_2,\eta), \eta \rangle )} \big)\|_{L^\infty}\leq C_\beta 2^{\frac{k|\beta'|}{2}}.
\end{eqnarray*}
Using $\eqref{L1-k-in}$, we have
$$
\sup_{\xi\neq0}|\xi|^{-1+|\alpha|} \|\partial^\alpha_\xi L^N b^\nu_{0k}(x,\xi,\eta) \|_{L^\infty}\lesssim 2^{k(\mii+2N(1-\rhii))}.
$$

Then for any integer $N\geq 0$, consider the operator $J=1-\partial^2_{\eta_1}-2^{-k}\partial^2_{\eta'}$, then
\begin{eqnarray*}
  |T^\nu_{0k}(x,y)|&\lesssim& 2^{kn} \big(1+g_k(y_2-\nabla_\eta \varphi_2(x_2,\eta^\nu))\big)^{-N} \\
  & & \ \cdot \bigg|\int \bigg(\int e^{i   \langle \nabla_\xi \varphi_1(x_1,\zeta_1)- y_1,\xi \rangle} e^{i ( 2^k \langle \nabla_\eta \varphi_2(x_2,\eta^\nu)-y_2, \eta \rangle )} J^N b^\nu_{0k}(x,\xi,\eta) d\xi\bigg) d\eta\bigg|\\
  &\lesssim& 2^{kn} \big(1+g_k(y_2-\nabla_\eta \varphi_2(x_2,\eta^\nu))\big)^{-N}\int_{T^\nu_k \cap A_\eta} \frac{2^{k(\mii+2N(1-\rhii))}}{\langle y_1-\nabla\varphi_1(x_1,\zeta_1)\rangle^{n+s}} d\eta \\
  &\lesssim&  \frac{2^{k(\mii+{n+1\over 2}+2N(1-\rhii))}}{\langle y_1-\nabla\varphi_1(x_1,\zeta_1)\rangle^{n+s}} \big(1+g_k(y_2-\nabla_\eta \varphi_2(x_2,\eta^\nu))\big)^{-N}
\end{eqnarray*}
for all $s\in [0,1)$, where we use the single parameter version of Lemma $\ref{le-b}$. With the similar interpolation argument as the previous case, we can conclude that for any real number $M>n$,
\begin{eqnarray*}
  |T^\nu_{0k}(x,y)|\lesssim \frac{2^{k(\mii+{n+1\over 2}+M(1-\rhii))}}{{\langle y_1-\nabla\varphi_1(x_1,\zeta_1)\rangle^{n+s}} } \big(1+g_k(y_2-\nabla_\eta \varphi_2(x_2,\eta^\nu))\big)^{-{M\over 2}},
\end{eqnarray*}
which implies
\begin{eqnarray*}
  \sup_x \int |T^\nu_{0k}(x,y)|dy\lesssim 2^{k(\mii+M(1-\rhii))},
\end{eqnarray*}
and
\begin{eqnarray*}
  \sup_y \int |T^\nu_{0k}(x,y)|dx\lesssim 2^{k(\mii+M(1-\rhii))}.
\end{eqnarray*}

Now we can conclude that
\begin{eqnarray*}
\sum_k \|T_{0k}^\nu f\|_{L^1}&\leq& \sum_k \sum_\nu \|T_{0k}^\nu f\|_{L^1} \lesssim \sum_k 2^{k(\mii+{n-1\over 2}+M(1-\rhii))}\|f\|_{L^1}\\
&\leq & \|f\|_{L^1}
\end{eqnarray*}
provided $\mii<-{n-1\over 2}-M(1-\rhii)$ and $M>n$.

\end{proof}

\subsection{Proof of  Theorem \ref{th-c}, the $L^2$ boundedness.}
\begin{proof}
 The $L^2$ boundedness of $T_{00}(f)$ follows from Theorem \ref{th-local}.

\vskip0.2cm

  Then we consider the case $T^{\mu,\nu}_{jk}$, we define an operator
  $$S^{\mu,\nu}_{jk}(\hat f):= T^{\mu,\nu}_{jk}(f),$$
  so it suffices to prove  $$\sum_{j,k}\sum_{\mu,\nu}\|S^{\mu,\nu}_{jk}( f)\|_{L^2}\lesssim \|f\|_{L^2}.$$

Instead, we can consider the operator $  S^{\mu,\nu}_{jk} (S^{\mu,\nu}_{jk})^* (f)$. Its kernel is given by

\begin{eqnarray*}
  & &R^{\mu,\nu}_{jk}(x,y)\\
  &=& 2^{(j+k)n} \int e^{i2^j(\varphi_1(x_1,\xi)-\varphi_1(y_1,\xi))} e^{i2^k(\varphi_2(x_2,\eta)-\varphi_2(y_2,\eta))}\\
  & &\ \cdot \chi^\mu_j(\xi)\chi^\nu_k(\eta)\Psi(\xi)\Psi(\eta) a(x,2^j\xi,2^k\eta)\bar a(y,2^j\xi,2^k\eta) d\xi d\eta \\
  &=& 2^{(j+k)n}\int_{\mathbb{R}^{2n}}  e^{i2^j\langle \nabla_\xi \varphi_1(x_1,\xi^\mu)-\nabla_\xi\varphi_1(y_1,\xi^\mu),\xi \rangle} e^{i2^k\langle\nabla_\eta \varphi_2(x_2,\eta^\nu)-\nabla_\eta \varphi_2(y_2,\eta^\nu),\eta \rangle} \\ & & \quad\cdot  b^{\mu,\nu}_{jk}(x,\xi,\eta)\bar b^{\mu,\nu}_{jk}(y,\xi,\eta)d\xi d\eta
\end{eqnarray*}

with
\begin{eqnarray*}
  b^{\mu,\nu}_{jk}(x,\xi,\eta)= e^{i2^j(\langle\nabla_\xi(x_1,\xi)-\nabla_\xi(x_1,\xi^\mu),\xi\rangle)} e^{i2^k(\langle \nabla_\eta(x_2,\eta)-\nabla_\eta(x_2,\eta^\nu),\eta\rangle )}\Psi(\xi) \Psi(\eta)\chi^\mu_j(\xi)  \chi^\nu_k(\eta) a(x,2^j\xi,2^k\eta).
\end{eqnarray*}

We consider the operator as before
$$L=(1-\partial^2_{\xi_1}-2^{-j}\partial^2_{\xi'})(1-\partial^2_{\eta_1}-2^{-k}\partial^2_{\eta'}).$$
Then as before we have
$$\sup_{\xi,\eta}\|L^N b^{\mu,\nu}_{jk}(\cdot,\xi,\eta)\|_{L^\infty}\lesssim 2^{j(\mi+2N(1-\rhi))} 2^{k(\mii+2N(1-\rhii))}.$$
Also, for all non-negative integers $N$, with $g_j(z)=2^{2j}z_1^2+2^j|z'|$ and $g_k(z)=2^{2k}z_1^2+2^k|z'|$,
\begin{gather*}
   |R^{\mu,\nu}_{jk}(x,y)|\lesssim  2^{(j+k)n}  \int\big| L^N(b^{\mu,\nu}_{jk}\big(x,\xi,\eta)\bar b^{\mu,\nu}_{jk}(y,\xi,\eta)\big)\big|  d\xi d\eta \\
  \cdot \big(1+g_j(\nabla_\xi \varphi_1(x_1,\xi^\mu)- \nabla_\xi \varphi_1(y_1,\xi^\mu))\big)^{-N}  \big(1+g_k(\nabla_\eta \varphi_2(x_2,\eta^\nu)-\nabla_\eta \varphi_2(y_2,\eta^\nu))\big)^{-N}.
\end{gather*}
Using the interpolation argument as before, we have for all real $M>0$,
\begin{gather*}
  |R^{\mu,\nu}_{jk}(x,y)| \lesssim 2^{j(2\mi+{n+1\over 2}+2M(1-\rhi))} 2^{k(2\mii+{n+1\over 2}+2M(1-\rhii))}\\
 \cdot \big(1+g_j(\nabla_\xi \varphi_1(x_1,\xi^\mu)- \nabla_\xi \varphi_1(y_1,\xi^\mu))\big)^{-M}  \big(1+g_k(\nabla_\eta \varphi_2(x_2,\eta^\nu)-\nabla_\eta \varphi_2(y_2,\eta^\nu))\big)^{-M}
\end{gather*}
With the non-smooth change of variable argument, it follows that
\begin{eqnarray*}
  \int |R^{\mu,\nu}_{jk}(x,y)| dy &\lesssim& 2^{j(2\mi+{n+1\over 2}+2M(1-\rhi))} 2^{k(2\mii+{n+1\over 2}+2M(1-\rhii))}\\
  & &\cdot \int_{A_\xi\cap T^\mu_j} \int_{A_\eta\cap T^\nu_k} (1+g_j(z_1))^{-M} (1+g_k(z_2))^{-M} dz_1 dz_2 \\
  & \lesssim & 2^{j(2\mi+2M(1-\rhi))} 2^{k(2\mii+2M(1-\rhii))}
\end{eqnarray*}
whenever $M>n/2.$

\vskip0.2cm

Note that $R_{jk}^{\mu,\nu}(x,y)$ is symmetric, we have
\begin{eqnarray}
&&\quad \|S_{jk}^{\mu,\nu}(S_{jk}^{\mu,\nu})^* (f)\|_{L^2}\leq  \bigg(\int\bigg\{\int |R_{jk}^{\mu,\nu}(x,y)|\cdot|f(y)|dy\bigg\}^2dx\bigg)^{\frac12}\nonumber\\
&&\leq\bigg\{\int \bigg(\int |R_{jk}^{\mu,\nu}(x,y)|dy\bigg)\cdot\bigg(\int |R_{jk}^{\mu,\nu}(x,y)|\cdot|f(y)|^2dy\bigg) dx\bigg\}^{\frac12} \nonumber \\
&&\leq\bigg(\sup_x\int |R_{jk}^{\mu,\nu}(x,y)|dy\bigg)^{\frac12}\cdot\bigg(\sup_y\int |R_{jk}^{\mu,\nu}(x,y)|dx\bigg)^{\frac12}\cdot\|f(y)\|_{L^2}\label{young} \\
&&=\sup_x\int |R_{jk}^{\mu,\nu}(x,y)|dy\cdot\|f\|_{L^2} \nonumber\\
&&\lesssim 2^{j(2\mi+2M(1-\rhi))} 2^{k(2\mii+2M(1-\rhii))} \|f\|_{L^2}\nonumber.
\end{eqnarray}
Then we can conclude
\begin{eqnarray*}
  \|T(f)\|_{L^2} &\leq &\sum_{j,k\geq 1}\sum_{\mu,\nu}\|T_{jk}^{\mu,\nu}(f)\|_{L^2}\leq \sum_{j,k\geq 1}\sum_{\mu,\nu} \left\{\|S_{jk}^{\mu,\nu}(S_{jk}^{\mu,\nu})^*(f)\|_{L^2}\|f\|_{L^2}\right\}^{1\over 2}\\
   &\lesssim &\sum_{j,k\geq 1} \left\{2^{j(2\mi+{n-1\over 2}+2M(1-\rhi))} 2^{k(2\mii+{n-1\over 2}+2M(1-\rhii))} \|f\|_{L^2}^2\right\}^{1\over 2}\\
   &\lesssim & \|f\|_{L^2}
\end{eqnarray*}
provided $\mi<-{n-1\over 4}+(\rhi-1)M,\mii<-{n-1\over 4}+(\rhii-1)M$ and $M>n/2$.

\vskip0.2cm

Now we turn to the estimate for $T^{\nu}_{0k}$, as before we define $S^\nu_{0k}(\hat f)= T^{\nu}_{0k}(f)$, and we will show  $$\sum_k \sum_\nu \| S^\nu_{0k}(f)\|_{L^2}\lesssim \|f\|_{L^2}.$$
Consider the kernel of the operator $ (S^\nu_{0k})( S^\nu_{0k})^* (f)$,
\begin{eqnarray*}
  & &R^{\nu}_{0k}(x,y)\\
  &=& 2^{kn} \int e^{i(\varphi_1(x_1,\xi)-\varphi_1(y_1,\xi))} e^{i2^k(\varphi_2(x_2,\eta)-\varphi_2(y_2,\eta))}\chi^\nu_k(\eta)\Psi_0(\xi)\Psi(\eta) a(x,\xi,2^k\eta)\bar a(y,\xi,2^k\eta) d\xi d\eta \\
  &=& 2^{kn}\int_{\mathbb{R}^{2n}} e^{i\langle\nabla_\xi \varphi_1(x_1,\zeta_1)-\nabla_\xi \varphi_1(y_1,\zeta_1),\xi \rangle} e^{i2^k\langle\nabla_\eta \varphi_2(x_2,\eta^\nu)-\nabla_\eta \varphi_2(y_2,\eta^\nu),\eta \rangle}  \\
  & & \qquad \qquad \cdot  b^{\nu}_{0k}(x,\xi,\eta)\bar b^{\nu}_{0k}(y,\xi,\eta)d\xi d\eta
\end{eqnarray*}

with
\begin{eqnarray*}
  b^{\nu}_{0k}(x,\xi,\eta)= e^{i \psi_1(x_1,\xi)} e^{i2^k(\langle \nabla_\eta(x_2,\eta)-\nabla_\eta(x_2,\eta^\nu),\eta\rangle )} \chi^\nu_k(\eta)\Psi_0(\xi)\Psi(\eta) a(x,\xi,2^k\eta).
\end{eqnarray*}
Define the operator $J=1-\partial^2_{\eta_1}-2^{-k}\partial^2_{\eta'}$, as before we have for $|\alpha|\geq 1$ and all integers $N\geq 0$,
\begin{equation}
\label{L2-k-in}
\sup_{\xi,\eta}|\xi|^{|\alpha|-1}\|\partial^\alpha_\xi J^N \big( b^{\nu}_{0k}\big)(x,\xi,\eta)\|_{L^\infty}\lesssim 2^{k(m+2N(1-\rhii))}.
\end{equation}
With $g_k(z)= 2^{2k}z_1^2+2^k |z'|^2$, $\eqref{L2-k-in}$ and lemma $\ref{le-b}$, we find that
\begin{eqnarray*}
  |R^{\nu}_{0k}(x,y)|
  &\lesssim& 2^{kn} \big(1+g_k(\nabla_\eta \varphi_2(x_2,\eta^\nu)-\nabla_\eta \varphi_2(y_2,\eta^\nu))\big)^{-N} \\
  & &\ \bigg|\int \bigg(\int e^{i\langle\nabla_\xi \varphi_1(x_1,\zeta_1)-\nabla_\xi \varphi_1(y_1,\zeta_1),\xi \rangle} e^{i2^k\langle\nabla_\eta \varphi_2(x_2,\eta^\nu)-\nabla_\eta \varphi_2(y_2,\eta^\nu),\eta \rangle}\\
   & & \quad \cdot J^N\big(b^{\nu}_{0k}(x,\xi,\eta)\bar b^{\nu}_{0k}(y,\xi,\eta)\big)  d\xi\bigg) d\eta\bigg|\\
  &\lesssim & 2^{kn} \big(1+g_k(\nabla_\eta \varphi_2(x_2,\eta^\nu)-\nabla_\eta \varphi_2(y_2,\eta^\nu))\big)^{-N} \\
   & & \quad \cdot \int_{A_\eta\cap\Gamma^\nu_k } \frac{ 2^{k(\mii+2N(1-\rhii))}}{\langle \nabla_\xi \varphi_1(x_1,\zeta_1)-\nabla_\xi \varphi_1(y_1,\zeta_1)\rangle^{n+s}}d\eta\\
  &\lesssim& \frac{2^{k(2\mii+{n+1\over 2}+2N(1-\rhii))} }{\langle \nabla_\xi \varphi_1(x_1,\zeta_1)-\nabla_\xi \varphi_1(y_1,\zeta_1)\rangle^{n+s}}\big(1+g_k(\nabla_\eta \varphi_2(x_2,\eta^\nu)-\nabla_\eta \varphi_2(y_2,\eta^\nu))\big)^{-N}
\end{eqnarray*}
holds for all $s\in[0,1)$ and $N$ replaced by any positive real number $M$ due to the interpolation. That implies
\begin{eqnarray*}
  \sup_x\int |R_{jk}^{\mu,\nu}(x,y)|dy<2^{k(2\mii +2N(1-\rhii))},\; \sup_y\int |R_{jk}^{\mu,\nu}(x,y)|dx<2^{k(2\mii+2N(1-\rhii))}
\end{eqnarray*}
for $M>n/2$, by using the non-smooth change of variables method.

\vskip0.2cm

As in \eqref{young}, we have

\begin{eqnarray*}
\|S_{0k}^{\nu}(S_{0k}^{\nu})^* (f)\|_{L^2} &\leq&\bigg(\sup_x\int |R_{jk}^{\mu,\nu}(x,y)|dy\bigg)^{\frac12}\cdot\bigg(\sup_y\int |R_{jk}^{\mu,\nu}(x,y)|dx\bigg)^{\frac12}\cdot\|f(y)\|_{L^2} \\
&\lesssim& 2^{k(2\mii+2N(1-\rhii))} \|f(y)\|_{L^2}.
\end{eqnarray*}

Then we can conclude
\begin{eqnarray*}
  \|T(f)\|_{L^2} &\leq &\sum_{k\geq 1}\sum_{\nu}\|T_{0k}^{\nu}(f)\|_{L^2}\leq \sum_{k\geq 1}\sum_{\nu} \left\{\|S_{0k}^{\nu}(S_{0k}^{\nu})^*(f)\|_{L^2}\|f\|_{L^2}\right\}^{1\over 2}\\
   &\leq &\sum_{k\geq 1} \left\{ 2^{k(2\mii+{n-1\over 2}+2M(1-\rhii))} \|f\|_{L^2}^2\right\}^{1\over 2}\\
   &\lesssim &  \|f\|_{L^2},
\end{eqnarray*}
provided $\mii<-{n-1\over 4}+(\rhii-1)M$ and $M>n/2$. Now we are done with the $L^2$ estimate.
\end{proof}

\subsection{Proof of Theorem \ref{th-d}}

\begin{proof}
Again, consider the decomposition used as in Section \ref{sec-decomp}, the boundedness of $T_{00}$ follows from Theorem $\ref{th-local}$. For $T_{jk}^{\mu,\nu}$ with $j,k\geq 1$, recall

\begin{eqnarray*}
 T_{jk}^{\mu,\nu}(f) =  2^{(j+k)n}\int_{\mathbb{R}^{2n}}  \chi_j^\mu(\xi)  \chi_k^\nu (\eta)e^{i (2^j\varphi_1(x,\xi)+2^k\varphi_1(x,\eta))}\Psi(\xi)\Psi(\eta) a(x,2^j\xi,2^k\eta)\widehat{f}(2^j\xi,2^k\eta)d\xi d\eta,
\end{eqnarray*}
where the kernel of the operator $T^{\mu,\nu}_{jk}$ is given by
\begin{eqnarray*}
  & &K^{\mu,\nu}_{jk}(x,y)\\
  &=& 2^{(j+k)n}\int_{\mathbb{R}^{2n}}  e^{i ( 2^j \langle \nabla_\xi \varphi_1(x_1,\xi^\mu)- y_1,\xi \rangle)} e^{i ( 2^k \langle \nabla_\eta \varphi_2(x_2,\eta^\nu)-y_2, \eta \rangle )} b^{\mu,\nu}_{jk}(x,\xi,\eta)d\xi d\eta,
\end{eqnarray*}
with
\begin{eqnarray*}
  b^{\mu,\nu}_{jk}(x,\xi,\eta)&=& \Psi(\xi) \Psi(\eta)\chi^\mu_j(\xi)  \chi^\nu_k(\eta) e^{i ( 2^j \langle \nabla_\xi \varphi_1(x_1,\xi)-\nabla_\xi \varphi_1(x_1,\xi^\mu),\xi \rangle)} \\
  & & \ \ \cdot e^{i ( 2^k \langle \nabla_\eta \varphi_2(x_2,\eta)-\nabla_\eta \varphi_2(x_2,\eta^\nu), \eta \rangle )} a(x,2^j\xi,2^k\eta).
\end{eqnarray*}
As the proof for $L^1$ case, with the following operator
$$L^N=(1-\partial^2_{\xi_1}-2^{-j}\partial^2_{\xi'})^{N}(1-\partial^2_{\eta_1}-2^{-k}\partial^2_{\eta'})^{N},\ N\in \mathbb{N},$$
we get
\begin{eqnarray*}
\sup_{\xi,\eta} \|L^N(b^{\mu,\nu}_{jk}(x,\xi,\eta))\|_{L^\infty}\leq C2^{j(m^{(1)}+2 N(1-\rho^{(1)}))}\cdot2^{k(m^{(2)}+2 N (1-\rho^{(2)}))}.
\end{eqnarray*}

To prove the desired $L^\infty$ boundedness, we only need to control $\int |K^{\mu,\nu}_{jk}(x,y)| dy$.
We can write
\begin{eqnarray*}
 & &L^N \big( e^{i ( 2^j \langle \nabla_\xi \varphi_1(x_1,\xi^\mu)- y_1,\xi \rangle)} e^{i ( 2^k \langle \nabla_\eta \varphi_2(x_2,\eta^\nu)-y_2, \eta \rangle )}\big) \\
 &=& (1+g_j(\nabla_\xi \varphi_1(x_1,\xi^\mu)- y_1))^{N}(1+g_k(\nabla_\eta \varphi_2(x_2,\eta^\nu)- y_2))^{N} \\
 & & \cdot e^{i ( 2^j \langle \nabla_\xi \varphi_1(x_1,\xi^\mu)- y_1,\xi \rangle)} e^{i ( 2^k \langle \nabla_\eta \varphi_2(x_2,\eta^\nu)-y_2, \eta \rangle)},
\end{eqnarray*}
where  $g_j(z)=2^{2j}z_1^2+2^j|z'|^2, g_k(z)=2^{2k}z_1^2+2^k|z'|^2$ for $  z\in \mathbb{R}^n.$

\vskip0.2cm

Now we can write
\begin{eqnarray*}
   & & \int |K^{\mu,\nu}_{jk}(x,y_1+\nabla_\xi\varphi_1(x_1,\xi^\mu),y_2+\nabla_\eta\varphi_1(x_2,\eta^\nu))|dy\\
   &=&  2^{(j+k)n} \int |\widetilde{b^{\mu,\nu}_{jk}}(x,y)|dy \\
   &=& 2^{(j+k)n} \big( \int_{\substack{g_j(y_1)\leq 2^{-2j\rhi}\\{g_k(y_2)\leq 2^{-2k\rhii}}}}+  \int_{\substack{g_j(y_1)\leq 2^{-2j\rhi}\\{g_k(y_2)\geq 2^{-2k\rhii}}}}+  \int_{\substack{g_j(y_1)\geq 2^{-2j\rhi}\\{g_k(y_2)\leq 2^{-2k\rhii}}}}+  \int_{\substack{g_j(y_1)\geq 2^{-2j\rhi}\\{g_k(y_2)\geq 2^{-2k\rhii}}}}\big) |\widetilde{b^{\mu,\nu}_{jk}}(x,y)|dy \\
   &=& 2^{(j+k)n}(I_1+I_2+I_3+I_4),
\end{eqnarray*}
where
$$\widetilde{b^{\mu,\nu}_{jk}}(x,y)=\int e^{-i2^j\langle y_1,\xi\rangle}  e^{-i2^k\langle y_2,\eta\rangle} b^{\mu,\nu}_{jk}(x,\xi,\eta) d\xi d\eta.$$

For $I_1$, using the Plancherel's theorem,
\begin{eqnarray*}
  I_1&\leq &\left(\int_{\substack{g_j(y_1)\leq 2^{-2j\rhi}\\{g_k(y_2)\leq 2^{-2k\rhii}}}} dy\right)^{\frac{1}{2}} \left(\int |\widetilde{b^{\mu,\nu}_{jk}}(x,y)|^2dy \right)^{\frac{1}{2}} \\
  &\lesssim & 2^{-(j+k){n+1\over 4}} 2^{-(j+k){n\over 2}}\left(\int_{\substack{|y_1|\leq 2^{-j\rhi}\\{|y_2|\leq 2^{-k\rhii}}}} dy\right)^{\frac{1}{2}} \left(\int |{b^{\mu,\nu}_{jk}}(x,\xi,\eta)|^2 d\xi d\eta\right)^{\frac{1}{2}} \\
  &\lesssim& 2^{-j({n+1\over 4}+{n\rhi\over 2}-\mi+{n-1\over 4}+{n\over 2})}2^{-k({n+1\over 4}+{n\rhii\over 2}-\mii+{n-1\over 4}+{n\over 2})}\\
  &\lesssim & 2^{-j(n-\mi+{n\rhi\over 2})}2^{-k(n-\mii+{n\rhii\over 2})}.
\end{eqnarray*}

For $I_4$, suppose $l>n/4$ is a non-negative integer, first consider the following estimate
\begin{eqnarray*}
  & &\left(\int |\widetilde{b^{\mu,\nu}_{jk}}(x,y)|^2 (1+g_j(y_1))^{2l}(1+g_k(y_2))^{2l}dy \right)^{1\over 2}\\
  &\lesssim& 2^{-{(j+k)n\over 2}}\left(\int |L^l b^{\mu,\nu}_{jk}(x,\xi,\eta) |^2 d\xi d\eta\right)^{1\over 2} \\
  &\lesssim& 2^{j(\mi+2l(1-\rhi)-{n-1\over 4}-{n\over 2})}2^{k(\mii+2l(1-\rhii)-{n-1\over 4}-{n\over 2})}.
\end{eqnarray*}
If $l$ is not an integer, we write $l=[l]+\{l\}$, where $[l]$ is the integer part and $\{l\}\in (0,1)$. Using the H\"older's inequality, we have
\begin{eqnarray*}
  & & \int |\widetilde{b^{\mu,\nu}_{jk}}(x,y)|^2 (1+g_j(y_1))^{2l}(1+g_k(y_2))^{2l}dy\\
  &=& \int |\widetilde{b^{\mu,\nu}_{jk}}(x,y)|^{2\{l\}} |\widetilde{b^{\mu,\nu}_{jk}}(x,y)|^{2-2\{l\}} (1+g_j(y_1))^{2\{l\}([l]+1)}(1+g_j(y_1))^{^{2[l](1-\{l\})}} \\ & & \quad\  \cdot(1+g_k(y_2))^{2\{l\}([l]+1)}(1+g_k(y_2))^{^{2[l](1-\{l\})}}dy\\
  &\leq& \left(\int |\widetilde{b^{\mu,\nu}_{jk}}(x,y)|^2  (1+g_j(y_1))^{2([l]+1)}(1+g_k(y_2))^{2([l]+1)}\right)^{\{l\}}\\
   & & \quad \ \cdot \left(\int |\widetilde{b^{\mu,\nu}_{jk}}(x,y)|^2  (1+g_j(y_1))^{2[l]}(1+g_k(y_2))^{2[l]}\right)^{1-\{l\}} \\
   &\lesssim&   2^{-{(j+k)n}} \left(\int |L^{[l]+1} b^{\mu,\nu}_{jk}(x,\xi,\eta) |^2 d\xi d\eta\right)^{\{l\}}\left( \int |L^{[l]} b^{\mu,\nu}_{jk}(x,\xi,\eta) |^2 d\xi d\eta\right)^{1-\{l\}}\\
  &\lesssim& 2^{2j(\mi+2l(1-\rhi)-{n-1\over 4}-{n\over 2})}2^{2k(\mii+2l(1-\rhii)-{n-1\over 4}-{n\over 2})}
  \end{eqnarray*}
Now we can get for any $l>{n\over 4}$,
\begin{eqnarray*}
  I_4&\leq &\left(\int_{\substack{g_j(y_1)\geq 2^{-2j\rhi}\\{g_k(y_2)\geq 2^{-2k\rhii}}}} (1+g_j(y_1))^{-2l} (1+g_k(y_2))^{-2l} dy\right)^{\frac{1}{2}} \\
   & &\quad \ \cdot \left(\int |\widetilde{b^{\mu,\nu}_{jk}}(x,y)|^2 (1+g_j(y_1))^{2l} (1+g_k(y_2))^{2l}dy \right)^{\frac{1}{2}} \\
  &\lesssim & 2^{-(j+k){n+1\over 4}}\left(\int_{\substack{|y_1|\geq 2^{-j\rhi}\\{|y_2|\geq 2^{-k\rhii}}}} |y_1|^{-4l} |y_2|^{-4l} dy\right)^{\frac{1}{2}} 2^{j(\mi+2l(1-\rhi)-{n-1\over 4}-{n\over 2})}2^{k(\mii+2l(1-\rhii)-{n-1\over 4}-{n\over 2})}  \\
  &\lesssim &  2^{-(j+k){n+1\over 4}} 2^{-j(\rhi({n\over 2}-2l))}2^{-k(\rhii({n\over 2}-2l))}2^{j(\mi+2l(1-\rhi)-{n-1\over 4}-{n\over 2})}2^{k(\mii+2l(1-\rhii)-{n-1\over 4}-{n\over 2})}\\
  &\lesssim & 2^{-j(n-\mi+{n\rhi\over 2}-2l)}2^{-k(n-\mii+{n\rhii\over 2}-2l)}.
\end{eqnarray*}

\vskip0.3cm
Now we estimate $I_2$. For any non-negative integer $l>n/4$, let $J^l=(1-\partial^2_{\eta_1}-2^{-k}\partial^2_{\eta'})^{l}$ we have
\begin{eqnarray*}
I_2&=&\int\limits_{\gfz{g_j(y_1)\leq2^{-2j\rho^{(1)}}}{g_k(y_2)>2^{-2k\rho^{(2)}}}}
|\widetilde{b^{\mu,\nu}_{j,k}}(x,y)|dy \\
&\leq&\bigg(\int\limits_{\gfz{g_j(y_1)\leq2^{-2j\rho^{(1)}}}{g_k(y_2)>2^{-2k\rho^{(2)}}}}(1+g_k(y_2))^{-2l}dy\bigg)^{\frac12}
\cdot\bigg(\int(1+g_k(y_2))^{2l}|\widetilde{b^{\mu,\nu}_{j,k}}(x,y)|^2dy\bigg)^{\frac12}\\
&\lesssim& 2^{-j\frac {(n+1)}{4}}2^{-j\frac {n\rho^{(1)}}2} 2^{-k\frac {(n+1)}{4}} 2^{-k(\rhii({n\over2}-2l))} \bigg(\int|\widetilde{J^l(b^{\mu,\nu}_{j,k})}(x,y)|^2dy\bigg)^{\frac12}\\
&\lesssim &2^{-j\frac {(n+1)}{4}}2^{-j\frac {n(\rho^{(1)})}2}2^{-j{n\over 2}} 2^{-k\frac {(n+1)}{4}} 2^{-k(\rhii({n\over2}-2l))}2^{-k{n\over 2}}\|J^l(b^{\mu,\nu}_{j,k})(x,\xi,\eta)\|_{L^2}\\
&\lesssim &2^{-j\frac {(n+1)}{4}}2^{-j\frac {n(\rho^{(1)})}2}2^{-j{n\over 2}} 2^{-j(-\mi+{n-1\over 4})}2^{-k\frac {(n+1)}{4}} 2^{-k(\rhii({n\over2}-2l))}2^{-k{n\over 2}}2^{-k(-\mii-2l(1-\rhii)+{n-1\over 4})}\\
&\lesssim & 2^{-j\big(n-m^{(1)}+\frac {n\rho^{(1)}}2\big)}2^{-k\big(n-m^{(2)}-2l+\frac {n\rho^{(2)}}2\big)}.
\end{eqnarray*}
As before, we then consider the case for non-integer $l$, and we use the notations as before.
\begin{eqnarray*}
I_2&=&\int\limits_{\gfz{g_j(y_1)\leq2^{-2j\rho^{(1)}}}{g_k(y_2)>2^{-2k\rho^{(2)}}}}
|\widetilde{b^{\mu,\nu}_{j,k}}(x,y)|dy \\
&\leq&2^{-j\frac {(n+1)}{4}}2^{-j\frac {n\rho^{(1)}}2} 2^{-k\frac {(n+1)}{4}} 2^{-k(\rhii({n\over2}-2l))}
 \bigg(\int(1+g_k(y_2))^{2l}|\widetilde{b^{\mu,\nu}_{j,k}}(x,y)|^2dy\bigg)^{\frac12}\\
 &\lesssim &2^{-j\frac {(n+1)}{4}}2^{-j\frac {n\rho^{(1)}}2} 2^{-k\frac {(n+1)}{4}} 2^{-k(\rhii({n\over2}-2l))}2^{-(j+k){n\over 2}}\\
 & & \quad \ \cdot
\bigg(\int|{J^{[l]+1}(b^{\mu,\nu}_{j,k})}(x,\xi,\eta)|^2dy\bigg)^{\frac{\{l\}}{2}}
\bigg(\int|{J^{[l]}(b^{\mu,\nu}_{j,k})}(x,\xi,\eta)|^2dy\bigg)^{\frac{1-\{l\}}{2}}\\
&\lesssim &2^{-j\frac {(n+1)}{4}}2^{-j\frac {n(\rho^{(1)})}2}2^{-j{n\over 2}} 2^{-j\{l\}(-\mi+{n-1\over 4})} 2^{-j(1-\{l\})(-\mi+{n-1\over 4})} 2^{-k\frac {(n+1)}{4}} 2^{-k(\rhii({n\over2}-2l))}2^{-k{n\over 2}} \\
& &\quad \ \cdot 2^{-k\{l\}(-\mii-2([l]+1)(1-\rhii)+{n-1\over 4})}2^{-k(1-\{l\})(-\mii-2[l](1-\rhii)+{n-1\over 4})}\\
&\lesssim & 2^{-j\big(n-m^{(1)}+\frac {n\rho^{(1)}}2\big)}2^{-k\big(n-m^{(2)}-2l+\frac {n\rho^{(2)}}2\big)}.
\end{eqnarray*}

Similarly, we can get that for all nonnegative real number $l>n/4$,
 $$I_3\lesssim 2^{-j\big(n-m^{(1)}-2l+\frac {n\rho^{(1)}}2\big)} 2^{-k\big(n-m^{(2)}+\frac {n\rho^{(2)}}2\big)}.$$

Then we have $$\sup_x \int |K^{\mu,\nu}_{jk}(x,y)| dy \lesssim 2^{-j(-\mi+{n\rhi\over 2}-2l)}2^{-k(-\mii+{n\rhii\over 2}-2l)}.$$
Taking the sum over $j,k,\mu,\nu$,
\begin{eqnarray*}
   \sum_{j,k} \sum_{\mu,\nu} \|T^{\mu,\nu}_{jk}\|_{L^\infty}
  &\lesssim& \sum_{j,k} 2^{-j(-\mi+{n\rhi\over 2}-2l-{n-1\over 2})}2^{-k(-\mii+{n\rhii\over 2}-2l-{n-1\over 2})}\|f\|_{L^\infty}\\
  &\lesssim& \|f\|_{L^\infty},
\end{eqnarray*}
provided $\mi<-{n-1\over 2}+{n\rhi \over 2}-2l,\, \mii<-{n-1\over 2}+{n\rhii \over 2}-2l $ and $l>{n\over 4}$.
\vskip0.3cm

Now it remains to consider the cases $T^{\nu}_{0k}(f)$.

\begin{eqnarray*}
 T_{0k}^\nu(f)= 2^{kn}\int_{\mathbb{R}^{2n}}   \chi_k^\nu e^{i (\varphi_1(x,\xi)+2^k\varphi_1(x,\eta))}\Psi_0(\xi)\Psi(\eta) a(x,\xi,2^k\eta)\widehat{f}(\xi,2^k\eta)d\xi d\eta
\end{eqnarray*}
where the kernel of the operator $T^{\nu}_{0k}$ is given by
\begin{eqnarray*}
K^{\nu}_{0k}(x,y)=2^{kn}\int_{\mathbb{R}^{2n}}  e^{i ( \langle \nabla_\xi \varphi_1(x_1,\zeta_1)- y_1,\xi \rangle)} e^{i ( 2^k \langle \nabla_\eta \varphi_2(x_2,\eta^\nu)-y_2, \eta \rangle )} b^{\nu}_{0k}(x,\xi,\eta)d\xi d\eta,
\end{eqnarray*}
with
\begin{eqnarray*}
  b^{\nu}_{0k}(x,\xi,\eta)= \Psi_0(\xi) \Psi(\eta) \chi^\nu_k(\eta) e^{i \psi_1(x_1,\xi)} e^{i ( 2^k \langle \nabla_\eta \varphi_2(x_2,\eta)-\nabla_\eta \varphi_2(x_2,\eta^\nu), \eta \rangle )} a(x,\xi,2^k\eta).
\end{eqnarray*}
for some $\zeta_1 \in\mathcal{S}^{n-1}$.

\vskip0.2cm

Also we have
\begin{eqnarray*}
 & &J^N \big( e^{i ( 2^k \langle \nabla_\eta \varphi_2(x_2,\eta^\nu)-y_2, \eta \rangle )}\big) =(1+g_k(\nabla_\eta \varphi_2(x_2,\eta^\nu)- y_2))^{N} e^{i ( 2^k \langle \nabla_\eta \varphi_2(x_2,\eta^\nu)-y_2, \eta \rangle )},
\end{eqnarray*}
where  $ g_k(z)=2^{2k}z_1^2+2^k|z'|^2$ for $  z\in \mathbb{R}^n.$

\vskip0.2cm

Now we can do the estimate
\begin{eqnarray*}
   & & \int |K^{\nu}_{0k}(x,y_1+\nabla_\xi\varphi_1(x_1,\zeta_1),y_2+\nabla_\eta\varphi_1(x_2,\eta^\nu))|dy\\
   &=& 2^{kn} \int |\widetilde{b^{\nu}_{0k}}(x,y)|dy \\
   &=& 2^{kn} \big( \int_{g_k(y_2)\leq 2^{-2k\rhii}}+ \int_{{g_k(y_2)\geq 2^{-2k\rhii}}}\big) |\widetilde{b^{\nu}_{0k}}(x,y)|dy \\
   &=& 2^{kn}(I_5+I_6),
\end{eqnarray*}
where
$$\widetilde{b^{\nu}_{0k}}(x,y)=\int e^{-i \langle y_1,\xi\rangle}  e^{-i 2^k\langle y_2,\eta\rangle} b^{\nu}_{0k}(x,\xi,\eta) d\xi d\eta.$$

Then there holds
\begin{eqnarray*}
  I_5&\leq &\int \left(\int_{g_k(y_2)\leq 2^{-2k\rhii}} dy\right)^{\frac{1}{2}} \left(\int |\widetilde{b^{\nu}_{0k}}(x,y)|^2dy_2 \right)^{\frac{1}{2}} d y_1\\
  &\lesssim & 2^{-k{n+1\over 4}} 2^{-k{n\over 2}}\int\left(\int_{{|y_2|\leq 2^{-k\rhii}}} dy\right)^{\frac{1}{2}} \left(\int_{A_\eta\cap \Gamma^\nu_k} |\int e^{i \langle y_1,\xi\rangle} {b^{\nu}_{0k}}(x,\xi,\eta)d\xi|^2 d\eta\right)^{\frac{1}{2}} dy_1.\\
  \end{eqnarray*}
Recall
$J^N=(1-\partial^2_{\eta_1}-2^{-k}\partial^2_{\eta'})^{N}, N \in \mathbb{N},$  and for multi-index $\alpha$ with $|\alpha|\geq 1$, we have
\begin{eqnarray*}
\sup_{\xi,\eta}|\xi|^{-1+|\alpha|} \|\partial^\alpha_\xi J^N(b^{\nu}_{0k}(x,\xi,\eta))\|_{L^\infty}\lesssim 2^{k(m^{(2)}+2 N (1-\rho^{(2)}))}.
\end{eqnarray*}
Therefore by Lemma \ref{le-b}
\begin{eqnarray}
\label{est-inf-5}
|\int e^{i \langle y_1,\xi\rangle} {b^{\nu}_{0k}}(x,\xi,\eta)d\xi|\lesssim \langle y_1\rangle^{-n-s}\cdot 2^{k\mii},\\
\label{est-inf-6}|\int e^{i \langle y_1,\xi\rangle} J^l {b^{\nu}_{0k}}(x,\xi,\eta)d\xi|\lesssim \langle y_1\rangle^{-n-s}\cdot 2^{k(\mii+2l(1-\rhii))},
\end{eqnarray}
for all $s\in [0,1)$ and all non-negative integer $l$.  \eqref{est-inf-5} implies
  \begin{eqnarray*}
  I_5 &\lesssim& 2^{-k({n+1\over 4}+{n\rhii\over 2}-\mii+{n-1\over 4}+{n\over 2})}\lesssim  2^{-k(n-\mii+{n\rhii\over 2})}.
\end{eqnarray*}

The other case, taking advantage of \eqref{est-inf-6}, it follows that for any non-negative integer $l>n/4$
\begin{eqnarray*}
  I_6&\leq & \int \left(\int_{g_k(y_2)\geq 2^{-2k\rhii}}(1+g_k(y_2))^{-2l} dy_2\right)^{\frac{1}{2}} \left(\int |\widetilde{b^{\nu}_{0k}}(x,y)|^2  (1+g_k(y_2))^{2l}dy_2 \right)^{\frac{1}{2}} dy_1 \\
  &\lesssim & 2^{-k{n+1\over 4}} 2^{-k{n\over 2}}\int \left(\int_{{|y_2|\geq 2^{-k\rhii}}} |y_2|^{-4l} dy_2\right)^{\frac{1}{2}} \left(\int |\int e^{i \langle y_1,\xi\rangle} {J^l b^{\nu}_{0k}}(x,\xi,\eta)d\xi|^2 d\eta\right)^{\frac{1}{2}} dy_1\\
  &\lesssim &  2^{-k{n+1\over 4}} 2^{-k(\rhii({n\over 2}-2l))}2^{k(\mii+2l(1-\rhii)-{n-1\over 4}-{n\over 2})}\\
  &\lesssim & 2^{-k(n-\mii+{n\rhii\over 2}-2l)}.
\end{eqnarray*}
When $l>n/4$ is not an integer, the same trick as before gives the same estimate
 \begin{eqnarray*}
   I_6 \lesssim 2^{-k(n-\mii+{n\rhii\over 2}-2l)}.
 \end{eqnarray*}

Then we have $$\sup_x \int |K^{\nu}_{0k}(x,y)| dy \lesssim 2^{-k(-\mii+{n\rhii\over 2}-2l)}.$$
Taking the sum over $k,\nu$,
\begin{eqnarray*}
  \sum_{k} \sum_{\nu} \|T^{\nu}_{0k}\|_{L^\infty} \lesssim \sum_{k} 2^{-k(-\mii+{n\rhii\over 2}-2l-{n-1\over 2})}\|f\|_{L^\infty}\lesssim\|f\|_{L^\infty},
\end{eqnarray*}
provided $\mii<-{n-1\over 2}+{n\rhii \over 2}-2l $ and any $l>{n\over 4}$.

\end{proof}

\subsection{Proof of Theorem \ref{th-e}}

\begin{proof}
(a)When $1\leq p\leq2$, assume $t$ satisfies
$$\frac1p=\frac{1-t}1+\frac t2=1-\frac t2 \ \ \Rightarrow1-t=\frac2p-1,\ \ t=2-\frac2p=2\bigg(1-\frac1p\bigg)$$
By the Riesz-Th\"{o}rin Interpolation, we know $\|T\|_{L^p\rightarrow L^p}\leq\|T\|^{1-t}_{L^1\rightarrow L^1}\cdot\|T\|^t_{L^2\rightarrow L^2}$. So, when
\begin{eqnarray*}
m^{(i)}&<&\bigg(-\frac{n-1}2+n(\rho^{(i)}-1)\bigg)\cdot \bigg(\frac2p-1\bigg)+\bigg(-{n-1\over 4}+\frac n2\cdot(\rho^{(i)}-1)\bigg)\cdot2\cdot\bigg(1-\frac1p\bigg)\\
&=&\frac{n(\rho^{(i)}-1)} {p}-\frac{(n-1)}{2p}
\end{eqnarray*}
we have that $T$ is bounded on $L^p$;\\

\vskip0.2cm

(b)When $2\leq p\leq\infty$, assume $t$ satisfies
$$\frac1p=\frac{1-t}2+\frac t\infty=\frac{1-t}2 \ \ \Rightarrow1-t=\frac2p,\ \ t=1-\frac2p$$
By the Riesz-Th$\ddot{o}$rin Interpolation, we know $\|T\|_{L^p\rightarrow L^p}\lesssim\|T\|^{1-t}_{L^2\rightarrow L^2}\cdot\|T\|^t_{L^\infty\rightarrow L^\infty}$, So, when
\begin{eqnarray*}
m^{(i)}&<&\left(-{n-1\over 4}+\frac n2\cdot(\rho^{(i)}-1)\right)\cdot\frac2p+\bigg(-\frac{n-1}2+\frac n2(\rho^{(i)}-1)\bigg)\cdot \bigg(1-\frac2p\bigg)\\
&=& \frac{n(\rho^{(i)}-1)}{2}-\frac{n-1}{2}\big(1-{1\over p}\big)
\end{eqnarray*}
we have that $T$ is bounded on $L^p$.
\end{proof}

{\bf Acknowledgement} We like to thank Chris Sogge for his valuable comments on  our work and for his encouragement 
to consider possible   applications of rough bi-parameter Fourier integral operators in PDEs in the future. 


\end{document}